\documentclass[a4paper,11pt]{article}

\usepackage{amsmath,amssymb,amsthm,amsfonts,setspace,hyperref,dsfont,yhmath}
\usepackage{latexsym}
\usepackage{amsmath,amsthm}
\usepackage{amssymb}
\usepackage{multicol}
\usepackage{graphics}
\usepackage{graphicx}

\newtheorem*{theoremA}{Theorem A}
\newtheorem*{theoremB}{Theorem B}
\newtheorem*{theoremC}{Theorem C}
\newtheorem*{theoremD}{Theorem D}
\newtheorem*{theoremE}{Theorem E}

\newtheorem*{corollaryA1}{Corollary A.1}
\newtheorem*{corollaryA2}{Corollary A.2}

\newtheorem*{corollaryB1}{Corollary B.1}
\newtheorem*{corollaryB2}{Corollary B.2}

\newtheorem*{corollaryD1}{Corollary D.1}

\newtheorem{theorem}{Theorem}[section]
\newtheorem{proposition}[theorem]{Proposition}
\newtheorem{lemma}[theorem]{Lemma}
\newtheorem{corollary}[theorem]{Corollary}
\newtheorem{remark}[theorem]{Remark}
\newtheorem{example}[theorem]{Example}
\newtheorem{question}[theorem]{Question}

\def\Aff{\mathop{\hbox{\rm Aff}}}
\def\Diff{\mathop{\hbox{\rm Diff}}}

\def\dim{\mathop{\hbox{\rm dim}}}

\def\id{{\rm id}}
\def\Id{{\rm Id}}

\def\top{{\rm top}}

\def\eqref#1{(\ref{#1})}

\def\disp{\displaystyle}
\def\<<{\prec}

\def\H{{\mathcal H}}

\def\H{{\mathcal H}}

\def\T{{\mathbb T}}
\def\Z{{\mathbb Z}}

\def\a{\alpha}

\def\e{\varepsilon}

\def\l{\lambda}
   \def\O{\Omega}

\newcommand{\NN}{\mathbb{N}}
\newcommand{\RR}{\mathbb{R}}

\newcommand{\CC}{\mathbb{C}}

\newcommand{\TT}{\mathbb{T}}
\newcommand{\ZZ}{\mathbb{Z}}

\newcommand{\norm}[1]{\lVert#1\rVert}
\newcommand{\abs}[1]{\lvert#1\rvert}
\begin{document}
\title{{
\bf Some ergodic and rigidity properties of discrete Heisenberg
group actions }  \footnotetext{The second
author is supported by Qing Lan Project of Jiangsu province, NSF
grands of Jiangsu province (No. BK2011275) and  NSFC (No.
11271278).} \footnotetext{The third author is supported by NSF
grants DMS-1346876.} \footnotetext{*Corresponding author.}
\author{Huyi Hu $^a$ \& Enhui Shi $^{b,*}$ \& 
Zhenqi Jenny Wang $^c$\vspace{3mm}\\
\small {\sl $^{a, c}$Department of Mathematics,}\\
\small{\sl Michigan State University, East Lansing, MI 48824,USA}\\
\small {\sl $^b$School of Mathematical Sciences,}\\
\small{\sl Soochow University, Suzhou, Jiangsu 215006, P.R.China}}
}
\maketitle


\begin{abstract}
The goal of this paper is to study ergodic and rigidity properties of 
smooth actions of the discrete Heisenberg group $\H$. We establish the
decomposition of the tangent space of any $C^\infty$ compact
Riemannian manifold $M$ for Lyapunov exponents, 
and show that all Lyapunov exponents for the center elements
are zero.  We obtain that if an $\H$ group action contains an Anosov element, 
then under certain conditions on the element, the center elements
are of finite order.  In particular there is no faithful codimensional 
one Anosov Heisenberg group action on any manifolds, and 
no faithful codimensional two Anosov Heisenberg group action on tori.  
In addition, we show smooth local rigidity for higher rank ergodic 
$\H$ actions by toral automorphisms, using a generalization of the KAM
(Kolmogorov-Arnold-Moser) iterative scheme.
\end{abstract}

\section{Introduction}

  In the past few decades, there is a considerable
progress in studying the ergodic theory and smooth rigidity of dynamical system of
higher rank abelian group actions. Smooth action (local) global rigidity for higher rank abelian algebraic actions has since been extensively studied; some of the highlights are \cite{Damjanovic3}, \cite{Damjanovic4} and \cite{KKH,HW}. We refer the reader to \cite{Sc} for a systematic
introduction of the dynamics of algebraic $\mathbb Z^d$ actions.
A natural question is how to extend these theories to noncommutative
group actions. The discrete Heisenberg group is a 2-step nilpotent
group, which is the most close to being abelian. So studying the
dynamical properties of discrete Heisenberg group actions is the
first step toward extending what we have known about $\mathbb Z^d$-actions. 

Throughout the paper, we use the symbol $\H$ to denote the
discrete Heisenberg group (see Section 1 for the explicit
definition). 
We first establish a tangent space decomposition into subspaces
related to Lyapunov exponents on any compact manifolds $M$ in Theorem A, 
which can be viewed as
an extension of the corresponding theorem established for $\mathbb
Z^2$ actions in \cite{Hu}. As corollaries of this theorem, we
obtain that the action of central elements of $\H$ must have $0$
Lyapunov exponents respect to any $\H$ invariant measure, and have
$0$ topological entropy when the action is $C^\infty$. This
indicates that the action of central elements in $\H$ cannot be
chaotic for any $\H$ action on compact manifolds.

The second part of our work is concerning faithfulness of $\H$ actions.
We show in Theorem B that if an $\H$ action is $C^r$, $r>1$,
and contains an Anosov element which has simple eigenvalues 
on stable direction with $\l_->\l_+^{\min\{r,2\}}$ (see \eqref{fdeflampm}),
then the action of any central element of $\H$ is of finite order.
Specially, it is true for any codimension $1$ action. 
For Anosov $\H$ actions on tori, we show further in Theorem D 
that the action of any central element
is either conjugate to a translation of finite order or conjugate
to an affine transformation of order $2$ when the action is of
codimension one or two. This implies specially that there is no
faithful Anosov $\H$ action on $\mathbb T^n$ with $n\leq 5$,
though there are faithful Anosov $\H$ actions on $\mathbb T^n$
with $n=6$ or $n\ge 8$ 
(see Example~\ref{Exam2} and Remark~\ref{RmkThmB} in Section~\ref{S1}). 

Lastly we obtain some regidity results for $\H$ actions on tori.
We prove that all such actions are topologically conjugate 
to an affine one in Theorem C, using the results in \cite{AP},
\cite{Fr} and \cite{Ma}.
Further, we extend an approach for proving local differentiable rigidity 
of Heisenberg group action by toral automorphisms, 
based on KAM-type iteration scheme that was first introduced 
in \cite{Damjanovic4} and was later developed in \cite{Damjanovic3}.

Recently, we note that the expansiveness and
homoclinic points for Heisenberg algebraic actions are
investigated by M. G\"oll, K. Schmidt, and E. Verbitskiy in
\cite{GSV}. One may consult \cite{GS, Li} for the study 
of abstract ergodic theory
about nilpotent group or amenable group actions.
It is known in 1970s that if 
$M={\mathbb R}, {\mathbb S}^1$ or $I=[0,1]$, then any nilpotent
subgroup of $\Diff^2(M)$ must be abelian (\cite{PT}), which
implies that there is no faithful $C^2$ action of $\H$ on $\mathbb
S^1$. In this century it was found out that
every finitely generated, torsion-free nilpotent
group has a faithful $C^1$ action on $M$ (\cite{FF}). For the case
$\dim M=2$, there are many faithful analytic Heisenberg group
actions on ${\mathbb S}^2$, closed disks, closed annulus and torus
(\cite{Pa}). However, Franks and Handel \cite{FH} showed that a
nilpotent group of $C^1$ diffeomorphisms which are isotopic to the
identity and preserve a measure whose support is all of $\mathbb T^2$ 
must be abelian. 

The paper is organized as following.
We state the results of the paper in Section~\ref{S1}.
Section~\ref{S2} is for proof of Thereom~A concerning Lyapunov exponents,
while Section~\ref{S3} is for proof of Thereom~B concerning faithfulness.
Thereom~C and D are proved in Section~\ref{S4}.
The last section is for smooth rigidity: the proof of Theorem E.


\section{Background and statement of result}\label{S1}

Let $M$ be a $C^\infty$ compact manifold. We denote by
$\Diff^n(M)$ the group of $C^r$ diffeomorphisms on $M$ for $r>0$. Let $\H$ be the discrete Heisenberg group and let $A$, $B$ and $\mathcal{C}$ denote the generators of $\H$: \
\begin{align}\label{for:1110}
A\mathcal{C}=\mathcal{C}A,\ B\mathcal{C}=\mathcal{C}B,\ AB=BA\mathcal{C}
\end{align}
Then for every $K\in \H$, there is
a unique triple $(n_1, n_2, n_3)\in \mathbb Z^3$ such that
$K=A^{n_1}B^{n_2}\mathcal{C}^{n_3}.$ Clearly, $\cal H$ is a 2-step nilpotent
group with center $\langle \mathcal{C}\rangle$. In this paper, we are
concerning ergodic and rigidity properties of $\cal H$ actions on compact
manifolds.

Let $\alpha: \H\rightarrow \Diff^n(M)$ be a $C^r$ group action of
$\H$ on a $C^\infty$ compact Riemannian manifold $M$, i.e.,
$\alpha: G\rightarrow {\rm Diff}^n(M)$ is a group homomorphism.
Write $f=\alpha(A)$, $g=\alpha(B)$ and $h=\alpha(\mathcal{C})$,
then
\begin{equation}\label{fdefHG}
fh=hf, \quad  gh=hg \quad  \mbox{and}\quad   fg=gfh.
\end{equation}
Throughout the paper, we always use $f$, $g$ and $h$ to denote
$\alpha(A)$, $\alpha(B)$ and $\alpha(\mathcal{C})$ respectively for a fixed
$\H$ action $\alpha$.

The Heisenberg group $\H$ naturally induces an action on $\mathbb
T^3$ since $\H$ embeds into $ {\rm SL}(3, \mathbb Z)$. We can obtain more
general examples as the following.

\begin{example}\label{Exam1}
Let
\[
\begin{array}{ccc}
A=\left[
  \begin{array}{ccc}
  X & I_n & O\\
  O & X & O\\
  O & O & X
  \end{array}
  \right], \quad \
\ B=\left[
  \begin{array}{ccc}
  Y & O & O\\
  O & Y & I_n\\
  O & O & Y
  \end{array}
  \right],  \quad \
\ \mathcal{C}=\left[
  \begin{array}{ccc}
  I_n & O & X^{-1}Y^{-1}\\
  O & I_n & O\\
  O & O & I_n
  \end{array}
  \right],
\end{array}
\]
where $X, Y\in {\rm  SL}(n, {\mathbb Z})$ with $XY=YX$, and $I_n$
and $O$ are $n\times n$ identity and zero matrix respectively.

It is easy to check that condition \eqref{for:1110} is satisfied. If
$M={\mathbb T}^{3n}$, then $A$, $B$ and $\mathcal{C}$ induce automorphisms
$f$, $g$ and $h$ on $M$ respectively that generate a Heisenberg
group action.
\end{example}


\subsection{Dynamical properties of the central element $h$}
A probability measure $\mu$ on $M$ is said to be {\it
$\alpha$-invariant} if $\alpha(k)_*\mu=\mu$ for every $k\in \H$.
We denote by ${\cal M}(M, \alpha)$ the set of all
$\alpha$-invariant Borel Probability measures. It is well known
and easy to prove that ${\cal M}(M, \a)={\cal M}(M, f)\cap{\cal
M}(M, g)\cap {\cal M}(M, h)\not=\emptyset$, where ${\cal M}(M, f)$
is the set of $f$-invariant probability measures.

Let $T_xM$ be the tangent space of $M$ at $x\in M$ and
$\phi\in{\Diff}^2(M)$, then $\phi$ induces a map $D\phi_x:
T_xM\rightarrow T_{\phi x}M$. It is well known that there exists a
measurable set $\Gamma_\phi$ with $\mu \Gamma_\phi=1$ for all
$\mu\in{\cal M}(M, \phi)$, such that for all $x\in\Gamma_\phi$,
$0\not=u\in T_xM$, the limit\vspace{-2mm}
$$
\chi(x,u,\phi)=\lim\limits_{n\to\infty}{1\over n}{\rm
log}\parallel D\phi_x^nu\parallel\vspace{-2mm}
$$
exists and is called the {\it Lyapunov exponent} of $u$ at $x$.

Let $\lambda_1(x,\phi)>\cdots>\lambda_{r(x,\phi)}(x,\phi)$ denote all
Lyapunov exponents of $\phi$ at $x$ with multiplicities
$m_1(x,\phi),\cdots,m_{r(x,\phi)}(x,\phi)$ respectively, and
$T_xM=\bigoplus_{i=1}^{r(x,\phi)}E_i(x,\phi)$ be the corresponding
decomposition of tangent space at $x\in M$.

The first result shows that the Lyapunov space decomposition of $\H$-action is similar to the
case of $\ZZ^2$ action:
\begin{theoremA}  There exists an $\H$-invariant
measurable set $\Gamma$ such that $\mu\Gamma=1$ for all
$\mu\in{\cal M}(M, \a)$, and for every $x\in\Gamma$ there is
a decomposition of the tangent space into
$$
T_xM=\bigoplus_{i=1}^{r(x,f)}\bigoplus_{j=1}^{r(x,g)}E_{ij}(x)
$$
such that if $E_{ij}(x)\not=\{0\}$, then for all $0\not=u\in
E_{ij}(x)$ and all $s,t,r\in\mathbb Z$,
\begin{equation}\label{fThmA}
\lim\limits_{n\to\infty}\frac{1}{n}{\rm log}\parallel
D(f^sg^th^r)_x^nu\parallel=s\lambda_i(x,f)+t\lambda_j(x,g).
\end{equation}
\end{theoremA}
By this theorem, we have some immediate corollaries, 
which indicate that the action of central elements in $\H$ cannot be
chaotic for any $\H$ action on compact manifolds. The concept of
chaos was first introduced by J.~Yorke and T.~Li in \cite{LJ}. Up to
today, there are many definitions of chaos based on different
viewpoints. In general, having positive Lyapunov exponents or
having positive topological entropy is regarded as an important
feature of chaos for diffeomorphisms.

\begin{corollaryA1}
All Lyapunov exponents of $h$ are zero with respect to
any measure $\mu\in{\cal M}(M, \a)$.
In particular, for any $n>0$, if $h$ has finitely many periodic points $p$
of period $n$, then all eigenvalues of $Dh^n_p$ have modulus $1$.
\end{corollaryA1}

\begin{corollaryA2}
If the action $\alpha$ is $C^\infty$, then the topological entropy
$h_{\top}(h)= 0$.
\end{corollaryA2}

\begin{remark}
As we mentioned in Introduction, Corollary~A.1 and Corollary~A.2
indicate that the central elements in $\H$ cannot be chaotic for
any $C^\infty$ $\H$ action on manifolds.
\end{remark}

Observe that $f^ng^n=g^nf^nh^{n^2}$, we have naturally the
following question.
\begin{question}
For any $\mu\in{\cal M}(M, \a)$, is $\|Dh^n_x v\|$ bounded by
$e^{\sqrt{n}\e}$ for some $\e>0$, or even by a polynomial in $n$
for $\mu$-a.e. $x$?
\end{question}

\subsection{Existence of faithful Anosov $\H$ action}
A group action $\alpha: G\to \Diff^r(M)$ is called \emph{faithful} 
if the map $\alpha$ is injective. 
We give some conditions under which an $\H$ action cannot be faithful.

A diffeomorphism $\phi$ on a compact manifold $M$ is called
\emph{Anosov}  if the whole tangent bundle has an uniformly
hyperbolic decomposition into $T_xM=E^s_x\oplus E^u_x$ invariant
under the differential $D\phi:TM\rightarrow TM$. For a group
action $\alpha: G\to \Diff^n(M)$, if  $\alpha(G)$ contains an
Anosov element, then the action $\alpha$ is called \emph{Anosov}.
The following is an example of Anosov $\H$ action:

\begin{example}\label{Exam2}
In Example~\ref{Exam1}, if $X$ or $Y$ is a hyperbolic matrix, then
$f$ or $g$ is a hyperbolic diffeomorphisms on $M=\T^{3n}$. Hence
$\a(\H)$ is an Anosov action. In particular, we can take
$M={\mathbb T}^{6}$, and
\[
\begin{array}{ccc}
X=Y=\left[
  \begin{array}{ccc}
  2 & 1\\
  1 & 1
  \end{array}
  \right], \quad
\ I_2=\left[
  \begin{array}{ccc}
  1 & 0 \\
  0 & 1
  \end{array}
\right],  \quad
\ O=\left[
  \begin{array}{ccc}
  0 & 0 \\
  0 & 0
  \end{array}
\right],
\end{array}
\]
and then we get a Anosov $\H$ action on $M={\mathbb T}^{6}$.
\end{example}

We say that an Anosov diffeomorphism $\phi$ has \emph{simple
eigenvalues on the stable direction} if for every periodic point
$p$ of period $n$, all eigenvalues of $D\phi^n_p|_{E^s_p}$ are
real and has algebraic multiplicity $1$. An Anosov diffeomorphism
is said to be of \emph{codimension $1$} if either $\dim E^u=1$ or $\dim
E^s=1$. Clearly, either $\phi$ or $\phi^{-1}$ has simple
eigenvalues on the stable direction if $\phi$ is of codimension
$1$.

If $D\phi_p^n|_{E^s_p}$ has eigenvalues $\l_1, \l_2, \cdots,
\l_{\dim E^s(\phi)}$ on the stable direction, we denote by
\begin{align}\label{fdeflampm}
  \l_-=\min_{1\le i\le \dim E^s(\phi)}|\l_i|, \quad
\l_+=\max_{1\le i\le \dim E^s(\phi)}|\l_i|.
\end{align}



\begin{theoremB} Let $\alpha: \H\rightarrow \Diff^r(M)$, $r>1$, be an
action of $\H$ on compact manifold $M$ such that $\alpha(\H)$
contains an  Anosov diffeomorphism that has simple eigenvalues
with $\l_->\l_+^{\min\{2, r\}}$ on the stable direction.  Then $h^k=\id$ for
some $k\ge 1$.
\end{theoremB}

Following from Theorem B, we have immediately

\begin{corollaryB1}
There is no faithful $C^r$, $r>1$, Heisenberg group action that contains
an Anosov element with simple eigenvalues with  $\l_->\l_+^{\min\{2, r\}}$ 
on the stable  direction. In particular, there is no faithful $C^r$
codimension $1$ Anosov Heisenberg group action on any compact
manifold.
\end{corollaryB1}

Also the proof of Theorem B gives

\begin{corollaryB2}
For a $C^r$, $r>1$, $\H$ action on $M$, 
if any element is a codimension 1 Anosov diffeomorphism 
that has exactly one fixed point, then $h^2=\id$.
\end{corollaryB2}

The next statement is a consequence of Theorem~D.
We state here since it concerns faithfulness.

\begin{corollaryD1}
There is no faithful $C^2$ Anosov Heisenberg group action on
$\mathbb T^n$ with $n\leq 5$.
\end{corollaryD1}

As supplements to Corollary~D.1, we give the following remark.

\begin{remark}\label{RmkThmB}
Example~\ref{Exam2} indicates that $\mathbb T^6$ does admit
faithful Anosov Heisenberg group $\cal H$ actions. So the number
``\ $5$" appeared in the above corollary is the best. In fact, it
is easy to see that there exist faithful Anosov Heisenberg group
actions on $\mathbb T^n$ for any integer $n\geq 8$ by using the action
in Example~\ref{Exam2} crossing an Anosov diffeomorphism of an manifold
of dimension two or higher.  However, the authors do not
know whether such group actions exist on $\mathbb T^7$. 
\end{remark}

\subsection{Rigidity of $\H$ action}
Let $\alpha, \alpha'$ be two actions of group $G$ on $M$, then
$\alpha$ and $\alpha'$ are said to be {\it topologically
conjugate} if there is a homeomorphism $T: M\rightarrow M$ such
that $T\circ\alpha(g)=\alpha'(g)\circ T$ for any $g\in G$.

The first two results are about global topological rigidity of an $\H$ action on $\TT^N$. Before the statement of Theorem C, we introduce some related notions.
Let $\mathbb T^n\equiv\mathbb R^n /\mathbb Z^n$ be the
$n$-dimensional torus. For any $A\in {\rm GL}(n, \mathbb Z)$ and
$a\in \mathbb R^n$, define
$T_{A, a}:\ \mathbb T^n\rightarrow\mathbb T^n$
by $T_{A, a}([x])=[Ax+a]$ for any $x\in\mathbb R^n$.
Such $T_{A, a}$ is called an \emph{affine transformation} on $\mathbb T^n$.
Specially, $T_{A,0}$ is called a \emph{linear automorphism}
of $\mathbb T^n$ induced by $A$ and $T_{\id, a}$ is called a
\emph{translation} on $\mathbb T^n$ by $a$.
All the affine transformations on $\mathbb T^n$ form a group
which is called the \emph{affine transformation group}
of $\T^n$ and is denoted by $\Aff(\mathbb T^n)$.  If $A$ has no
eigenvalue with modular $1$, then $T_{A,a}$ is an Anosov
diffeomorphism.
\begin{theoremC}
Every Anosov Heisenberg group action on ${\mathbb T^n}$
is topologically conjugate to an affine one.
\end{theoremC}
In some cases, the form of $h$ can be completely determined as the
following theorem showed.

\begin{theoremD}
If  $f$ is a codimension $1$ Anosov diffeomorphism of a
Heisenberg group action on ${\mathbb T^n}$, then $h$ is
topologically conjugate to a translation of finite order. If $f$
is a  codimension $2$ Anosov diffeomorphism for a Heisenberg group
action on ${\mathbb T^n}$, then $h$ is either topologically
conjugate to a translation of finite order or topologically
conjugate to an affine transformation $T_{-\id, c}$ of order $2$
for some $c\in\mathbb T^n$.
\end{theoremD}
For the codimension $1$ case, the following example indicates that
$h$ in the above theorem can be non-trivial.

\begin{example}\label{Exam3}
Let $\disp A=\left[
  \begin{array}{cc}
  5 & 3 \\ 3 & 2
  \end{array}
  \right]$,
$\disp b=\left(
  \begin{array}{c}
  2/5  \\ 3/5
  \end{array}
  \right)$,
$\disp c=\left(
  \begin{array}{c}
  2/5  \\ 4/5
  \end{array}
  \right)$.
Define affine transformations $f, g, h$ on ${\mathbb T}^2$ by
$f([x])=[Ax]$, $g([x])=[Ax+b]$, and $h([x])=[x+c]$ for all $x\in
\mathbb R^2$. Then $fh=hf$, $gh=hg$, $fg=gfh$, and $h^5([x])=[x]$.
Thus we get a Heisenberg group action on $\mathbb T^2$ with $h$
being a translation of order~$5$.
\end{example}
The following example shows that there do exist examples such that
$h$ is conjugate to $T_{-Id, c}$ as shown in Theorem D in the
codimension $2$ case.

\begin{example}\label{Exam4}
Take $\disp X=\left[
  \begin{array}{cc}
  2 & 1 \\ 1 & 1
  \end{array}
  \right]$. Let
  $\disp A=\left[
  \begin{array}{cc}
  X & 0 \\ 0 & -X
  \end{array}
  \right]$,
$\disp B=\left[
  \begin{array}{cc}
  0 & X\\ X & 0
  \end{array}
  \right]$, and
$\disp C=\left[
  \begin{array}{cc}
  -I & 0 \\ 0 & -I
  \end{array}
  \right]$,
where $I$ is the $2\times 2$ identity matrix. For any $c\in\mathbb
R^4$, let $a=-{1\over 2}(A-I)c$ and $b=-{1\over 2}(B-I)c$. Then it
is easy to check that the affine transformations $T_{A,a}, T_{B,
b}$ and $T_{-\id,c}$ on $\mathbb T^4$ satisfy the relations
$T_{A,a}T_{-\id, c}=T_{-\id, c}T_{A,a}$, $T_{B,b}T_{-\id,
c}=T_{-\id, c}T_{B,b}$ and $T_{A,a}T_{B, b}=T_{B,
b}T_{A,a}T_{-\id, c}$. Thus we get an affine Anosov $\cal H$
action on $\mathbb T^4$ with $h$ being of the form $T_{-\id, c}$.
\end{example}

An action $\alpha$ of a finitely generated discrete group $G$ on a
manifold $M$ is $C^{k,r,\ell}$ \emph{locally rigid} if any $C^k$
perturbation $\widetilde{\alpha}$ which is sufficiently $C^r$
close to $\alpha$ on a finite generating set is $C^\ell$ conjugate
to $\alpha$; i.e., there exists a diffeomorphism $T$ of $M$ 
$C^\ell$ close to identity which conjugates $\widetilde{\alpha}$ to
$\alpha$: $T\circ\alpha(g)=\alpha'(g)\circ T$ for any $g\in G$.

An $\H$ action $\alpha$ by automorphisms on $\TT^N$ is called an
\emph{ergodic higher rank action}  if it contains two elements
$h_1,\,h_2$ such that $\alpha(h_1^mh_2^n)\in SL(N,\ZZ)$ is ergodic
for all $(m,n)\neq 0$ in $\ZZ^2$.

\begin{theoremE}
Let $\alpha$ be an ergodic higher rank $\H$ action by
automorphisms of the $N$-dimensional torus. Then there exists a
constant $l=l(\alpha,N)\in\NN$ such that $\alpha$ is
$C^{\infty,l,\infty}$ locally rigid.
\end{theoremE}

\section{Lyapunov exponents: Proof of Theorem A and its corollaries}\label{S2}
\setcounter{equation}{0}

Since $g$ and $h$ are commuting maps, by Theorem A in \cite{Hu}
there exists a measurable set $\Gamma_0$ with
$g^sh^t\Gamma_0=\Gamma_0$\ $\forall s, t\in\mathbb Z$,  and
$\mu\Gamma=1$, $\forall \mu\in{\cal M}(M, g, h)$, such that for all
$x\in \Gamma_0$, there is a (unique) decomposition of the tangent
space into
\begin{equation}\label{fSA01}
T_xM=\bigoplus_{j=1}^{r(x,g)}\bigoplus_{k=1}^{r(x,h)}E_{jk}(x)
\end{equation}
such that for all $s, t\in\mathbb Z$ with $E_{jk}(x)\not=0$,
for all  $0\not=u\in E_{jk}(x)$,
\begin{equation}\label{fSA02}
\lim\limits_{n\to\infty}\frac{1}{n}{\rm log}\parallel
D(g^sh^t)_x^nu\parallel=s\lambda_j(x,g)+t\lambda_k(x,h).
\end{equation}
Moreover,
\begin{equation*}
D(g^sh^t)_x(E_{jk}(x))=E_{jk}(g^sh^tx)
\end{equation*}
and
\begin{equation*}
\lambda_j(g^sh^tx, g)=\lambda_j(x, g), \quad
\lambda_k(g^sh^tx, h)=\lambda_k(x, h).
\end{equation*}

Let $\Gamma_1=\cap_{i\in\mathbb Z}f^i\Gamma_0$. By \eqref{fdefHG}
it is easy to see that $f\Gamma_1=\Gamma_1$,
$g\Gamma_1=\Gamma_1$ and $h\Gamma_1=\Gamma_1$. So $\Gamma_1$ is an
$\H$-invariant measurable set and $\mu\Gamma_1=1$ \
$\forall\mu\in{\cal M}(M, \a)$.

\begin{lemma}\label{LThmA1}
For all $x\in \Gamma_1$ and all $0\not=u\in E_{jk}(x)$, we have
$$
\chi(fx, Df_xu, g^{\pm})=\pm\lambda_j(x, g)\mp\lambda_k(x, h),\ \
\chi(fx, Df_xu, h^{\pm})=\pm\lambda_k(x, h).
$$
\end{lemma}

\begin{proof}
By \eqref{fdefHG} we have
$$
\parallel Dg^{\pm n}Df_xu\parallel
=\parallel DfDg^{\pm n}Dh_x^{\mp n}u\parallel,\ \
\parallel Dh^{\pm n}Df_xu\parallel=\|DfDh_x^{\pm n}u\|.
$$
So there is a constant $c>0$ such that for all $n\geq 0$,
\begin{equation*}
\begin{split}
  c^{-1}\|Dg^{\pm n}Dh_x^{\mp n}u\|
\leq& \|Dg^{\pm n}Df_x u\|
\leq c\|Dg^{\pm n}Dh_x^{\mp n}u\|,  \\
  c^{-1}\|Dh_x^{\pm n}u\|
\leq& \|Dh^{\pm n}Df_xu\|
\leq c\|Dh_x^{\pm n}u\|.
\end{split}
\end{equation*}
Then by \eqref{fSA02} we get that
\begin{equation}\label{fThmAL1}
\begin{split}
 \chi(fx, Df_xu, g^{\pm 1})
=&\lim\limits_{n\to \infty}\frac{1}{n}\log\| Dg^{\pm n}Df_xu\|
=\pm\lambda_j(x, g)\mp\lambda_k(x, h),  \\
\chi(fx, Df_xu, h^{\pm 1})
=&\lim\limits_{n\to\infty}\frac{1}{n}\log\|Dh^{\pm n}Df_xu\|
=\pm\lambda_k(x, h),
\end{split}
\end{equation}
which are what we need.
\end{proof}

\begin{lemma}\label{LThmA2}
For any $x\in \Gamma_1$, any $(j,
k)$ with $E_{j, k}\not=0$ and any $0\not=u\in E_{j, k}$, we have
$\chi(x, u, h)=0$.
\end{lemma}

\begin{proof}
Assume to the contrary that there is some
$x_0\in\Gamma_1$ and some $(j_0, k_0)$ with $E_{j_0, k_0}\not=0$
and some $0\not=u_0\in E_{j_0, k_0}$ such that $\chi(x_0, u_0,
h)\not=0$. From Lemma~\ref{LThmA1} we have
\begin{equation*}
\begin{split}
  \chi(fx_0, Df_{x_0}u_0, g)
=&\lambda_{j_0}(x_0, g)-\lambda_{k_0}(x_0, h), \\
  \chi(fx_0, Df_{x_0}u_0, h^{\pm 1})
=&\pm \lambda_{k_0}(x_0, h).
\end{split}
\end{equation*}
It follows that in the decomposition
$$
T_{fx_0}M
=\bigoplus_{j=1}^{r(fx_0, g)}\bigoplus_{k=1}^{r(fx_0,h)}E_{jk}(fx_0),
$$
there is some $E_{j_1k_1}(fx_0)\not=0$ such that for all
$0\not=u\in E_{j_1k_1}(fx_0)$,
\begin{equation*}
\begin{split}
\chi(fx_0, u, g)=&\lambda_{j_0}(x_0, g)-\lambda_{k_0}(x_0,h), \\
\chi(fx_0, u, h)=&\lambda_{k_0}(x_0, h).
\end{split}
\end{equation*}
Then by induction process, we get that in the decomposition
$$
T_{f^nx_0}M
=\bigoplus_{j=1}^{r(f^nx_0,g)}\bigoplus_{k=1}^{r(f^nx_0, h)}E_{jk}(f^nx_0),
$$
there is some $E_{j_nk_n}\not=0$ such that for all $0\not=u\in E_{j_nk_n}$,
\begin{equation*}
\begin{split}
\chi(f^nx_0, u, g)
=& \lambda_{j_0}(x_0, g)-n\lambda_{k_0}(x_0,h), \\
\chi(f^nx_0, u, h)
=&\lambda_{k_0}(x_0, h).
\end{split}
\end{equation*}
Since
$\lambda_{k_0}(x_0, h)=\chi(x_0, u_0, h)\not=0,$\  $|\chi(f^nx_0,
u, g)|\rightarrow\infty$ as $n\rightarrow\infty$, contradicting
to the fact that $|\chi(p, v, g)|\leq\sup_{x\in M}\|Dg_x\|<\infty$\
$\forall p\in \Gamma_1$, $\forall v\in T_pM$.
\end{proof}

\begin{proposition}\label{PThmA3}
For all $x\in \Gamma_1$, there is a decomposition of the tangent space
into $T_xM=\bigoplus_{j=1}^{r(x, g)}E_j(x, g)$
such that for all $0\not=u\in E_j(x, g)$,

\begin{enumerate}
\item[{\rm (i)}]
the spectrum $\{\lambda_j(x, g),\ m_j(x, g),\ j=1, \cdots, r(x, g)\}$
is $\H$-invariant;

\item[{\rm (ii)}]
$D(f^sg^th^r)_xE_j(x, g)=E_j(f^sg^th^rx, g)$ \
$\forall j=1, \cdots, r(x, g)$,\ $\forall\; s, t, r\in\mathbb Z$;

\item[{\rm (iii)}]
$\disp \lim\limits_{n\rightarrow\infty}\frac{1}{n}\log\|D(g^th^r)_x^nu\|
=t\lambda_j(x, g)$ \ $\forall \; t, r\in\mathbb Z$.
\end{enumerate}
\end{proposition}

\begin{proof}
By Lemma~\ref{LThmA2}, we get that the number
$r(x, h)=1$ and the decomposition~\eqref{fSA01} becomes that,
for all $x\in \Gamma_1$,
$$
T_xM=\bigoplus_{j=1}^{r(x, g)}E_j(x, g), 
$$
where $E_j(x, g)=\bigoplus_{k=1}^{r(x, h)}E_{jk}(x)=E_{j1}(x)$.
It follows from Lemma~\ref{LThmA2} again that equations~\eqref{fThmAL1}
become
\begin{equation*}
\begin{split}
\chi(fx, Df_xu, g^{\pm 1})=&\pm\lambda_j(x, g),  \\
\chi(fx, Df_xu, h^{\pm 1})=&0.
\end{split}
\end{equation*}
So (i) and (ii) holds, and (iii) follows from \eqref{fSA02}.
\end{proof}

Exchanging $f$ and $g$, we can get the following proposition
similarly:

\begin{proposition}\label{PThmA4}
There is an $\H$-invariant measurable set $\Gamma_2$ with $\mu \Gamma_2=1$, \
$\forall \mu\in {\cal {M}}(M, \a)$ such that for all $x\in\Gamma_2$ there is
a decomposition of the tangent space into
$T_x=\bigoplus_{i=1}^{r(x, f)}E_i(x, f)$
satisfying that for every $0\not=u\in E_i(x, f)$,
\begin{enumerate}
\item[{\rm (i)}]
the spectrum $\{\lambda_i(x, f),\ m_i(x, f),\ i=1,
\cdots, r(x, f)\}$ is  $\H$-invariant;

\item[{\rm (ii)}]
$D(f^sg^th^r)_xE_i(x, f)=E_i(f^sg^th^rx, f)$\;
$\forall i=1, \cdots, r(x, f)$, $\forall\; s, t, r\in\mathbb Z$;

\item[{\rm (iii)}]
$\disp \lim\limits_{n\rightarrow\infty}\frac{1}{n}\log\|D(f^sh^r)_x^nu\|
=s\lambda_i(x, f)$, \ $\forall\; s, r\in\mathbb Z$.
\end{enumerate}
\end{proposition}

\begin{proposition}\label{PThmA5}
There is an $\H$-invariant measurable set
$\Gamma_3\subset \Gamma_1\cap \Gamma_2$ with $\mu \Gamma_3=1$\
$\forall \mu\in {\cal M}(M, \a)$ such that for
all $x\in \Gamma_3$ there is a decomposition of the tangent space into
$$
T_xM=\bigoplus_{i=1}^{r(x, f)}\bigoplus_{j=1}^{r(x, g)}E_{ij}(x)
$$
satisfying that if $E_{ij}(x)\not=0$, then for all
$0\not=u\in E_{ij}(x)$ and all $s, t, r\in\mathbb Z$,

\begin{enumerate}
\item[{\rm (i)}]
$\disp \lim_{n\to\infty}\frac{1}{n}\log\|D(f^sh^r)_x^nu\|=s\lambda_i(x, f)$,
$\disp \lim_{n\to\infty}\frac{1}{n}\log\|D(g^th^r)_x^nu\|=t\lambda_j(x, g)$;

\item[{\rm (ii)}]
$D(f^sg^th^r)_xE_{ij}(x)=E_{ij}(f^sg^th^rx)$;

\item[{\rm (iii)}]
$\lambda_i(f^sg^th^rx, f)=\lambda_i(x, f)$, \
$\lambda_j(f^sg^th^rx, g)=\lambda_j(x, g)$.
\end{enumerate}
\end{proposition}

\begin{proof}
For all $x\in \Gamma_1\cap\Gamma_2$, let
$T_xM=\bigoplus_{j=1}^{r(x, g)}E_j(x, g)$ be the
decomposition given in Proposition~\ref{PThmA3}.
By Proposition~\ref{PThmA3}(ii), $Df_x(E_j(x, g))=E_j(fx, g)$.
Restricted to $\{E_j(x, g)\}$, $\{Df_x^n\}$ is a cocycle on $M$
with respect to $f$.
Thus, similar to the proof of Proposition~2.3 in \cite{Hu},
we obtain an $\H$-invariant measurable set
$\Gamma_3\subset\Gamma_1\cap\Gamma_2$
with $\mu\Gamma_3=1$ \ $\forall \mu\in {\cal M}(M, \a)$, and a
decomposition of the tangent space into
$$
T_xM=\bigoplus_{i=1}^{r(x, f)}\bigoplus_{j=1}^{r(x, g)}E_{ij}(x), \
\forall  x\in\Gamma_3.
$$
Clearly $E_{ij}(x)=E_i(x, f)\cap E_j(x, g)$.  Thus (i) (ii) and
(iii) are direct corollaries of Proposition~\ref{PThmA3} and
Proposition~\ref{PThmA4}.
\end{proof}

\begin{proof}[Proof of Theoren A]
For any $s, t, r\in\mathbb Z$ and any $\varepsilon>0$, set
\begin{equation*}
\begin{split}
A^+_\varepsilon
=&\{x: \exists 0\not=u\in\! E_{ij}(x)\;  \mbox{s.t.}\;
     \chi(x, u, f^sg^th^r)-s\lambda_i(x, f)-t\lambda_j(x,g)
>(|\lambda|+1)\varepsilon\},              \\
A^-_\varepsilon
=&\{x: \exists 0\not=u\in\! E_{ij}(x)\;  \mbox{s.t.}\;
     \chi(x, u, f^sg^th^r)-s\lambda_i(x, f)-t\lambda_j(x,g)
<(|\lambda|+1)\varepsilon\},
\end{split}
\end{equation*}
where $\lambda=6s+6t+|2r-st|$.
It is easy to see that we only need to prove that for any
$\mu\in {\cal M}(M, \a)$, for all
$\varepsilon>0$, $\mu(A^\pm_\varepsilon)=0$

Now we prove $\mu(A^+_\varepsilon)=0$,
the other one can be obtained similarly.

Suppose $\mu(A^+_\varepsilon)>0$ for some $\mu\in {\cal M}(M, \a)$
and $\varepsilon>0$.  Then there exists a constant $C>0$
sufficiently large such that the set
\begin{equation}\label{fThmA1}
\begin{split}
  A_{\varepsilon, C}
:=\{x: \  &\exists\; 0\not=u\in E_{ij}(x)\  \mbox{s.t.}\
  \|D(f^sg^th^r)_x^nu\|      \\
 >&C^{-1}\|u\|\exp n(s\lambda_i(x,f)+t\lambda_j(x, g)+|\lambda|\varepsilon)\
   \forall n\geq 0\}
\end{split}
\end{equation}
satisfies $\mu(A_{\varepsilon, C})>0$.
Let
\begin{equation*}
\delta=\mu(A_{\varepsilon, C}).
\end{equation*}
By \eqref{fdefHG} we have
\begin{equation*}
\|D(f^sg^th^r)_x^nu\|=\|Dh^{st\frac{n(n-1)}{2}}Df^{sn}Dg^{tn}Dh_x^{rn}u\|.
\end{equation*}
Then
$$
\|D(f^sg^th^r)_x^{2n}u\|=\|Dh^{2stn^2}Df^{2sn}Dg^{2tn}Dh_x^{(2r-st)n}u\|
\quad \forall n\geq 0.
$$
For $l>0$, let
$$
A^l_f=\{x: l^{-1}\|u\|\exp n(\lambda_i(x, f)-\varepsilon)
\leq\|Df_x^nu\|
\leq l\|u\|\exp n(\lambda_i(x,f)+\varepsilon) \;
\forall u\in E_{ij}(x)\; \forall n\geq 0\}.
$$
$$
A^l_g=\{x: l^{-1}\|u\|\exp n(\lambda_j(x,g)-\varepsilon)
\leq\|Dg_x^nu\|
\leq l\|u\|\exp n(\lambda_j(x,g)+\varepsilon) \;
\forall u\in E_{ij}(x)\; \forall n\geq 0\}.
$$
$$
A^l_h=\{x: l^{-1}\|u\|\exp(-|n|\varepsilon)
\leq\|Dh_x^nu\|
\leq l\|u\|\exp (|n|\varepsilon) \;
\forall u\in E_{ij}(x)\;  \forall n\in\mathbb Z\}.
$$
Choose $l$ sufficiently large so that
$$
\mu(A^l_i)>1-\frac{1}{26}\delta, \quad i=f, g, h.
$$
Let
$$
B_n=A^l_g\cap g^{-tn}(A^l_f)\cap A^l_h\cap
    h^{-stn^2}f^{-sn}(A^l_g)\cap h^{-stn^2}(A^l_f).
$$
Then
$\mu(B_n)>1-\frac{5}{26}\delta,$ and for all $x\in B_n$ and all
$0\not=u\in E_{ij}(x)$, we have
\begin{equation}\label{fThmA2}
\begin{split}
     &\|D(f^{sn}g^{tn})_xu\|
\leq  l\|Dg_x^{tn}u\| \exp sn(\lambda_i(g^{tn}x, f)+\varepsilon) \\
\leq & l^2\|u\|
   \exp\bigl[tn\lambda_j(x, g)+sn\lambda_i(x,f)+(s+t)n\varepsilon\bigr],
\end{split}
\end{equation}
and
\begin{equation}\label{fThmA3}
\begin{split}
     & \|D(g^{tn}f^{sn}h^{stn^2})_xu\|
\geq  l^{-1}\|D(f^{sn}h^{stn^2})_xu\|\exp(tn\lambda_j(x,g)-tn\varepsilon) \\
\geq & l^{-2}\|Dh^{stn^2}u\|
   \exp\bigl[sn\lambda_i(x,f)+tn\lambda_j(x, g)-(t+s)n\varepsilon\bigr].
\end{split}
\end{equation}
Since
$f^{sn}g^{tn}=g^{tn}f^{sn}h^{stn^2}$, it follows from \eqref{fThmA2} and
\eqref{fThmA3} that
$$
\|Dh^{stn^2}u\|
\leq l^4\|u\|\exp[2(t+s)n\varepsilon] \quad
\forall n\geq 0,\ \forall x\in B_n,\ \forall 0\not=u\in E_{ij}(x).
$$
Let $C_n=h^{-stn^2}(B_n)\cap B_n$.
Then $\mu(C_n)>1-\frac{10}{26}\delta$.
For all $x\in C_n$ and all $0\not=u\in E_{ij}(x)$, we have
$$
\|Dh_x^{2stn^2}u\|=\|Dh^{stn^2}Dh_x^{stn^2}u\|
\leq l^4\|Dh_x^{stn^2}u\|e^{2(s+t)n\varepsilon}
\leq l^8\|u\|e^{4(s+t)n\varepsilon}.
$$
Let
$$
D_n=h^{-(2r-st)n}g^{-2tn}f^{-2sn}(C_n)\cap
h^{-(2r-st)n}g^{-2tn}(A^l_f)\cap h^{-(2r-st)n}(A^l_g)\cap A^l_h.
$$
Then
$$
\mu(D_n)>1-\frac{10}{26}\delta-\frac{3}{26}\delta
=1-\frac{\delta}{2}>1-\delta,
$$
and so
$$
\mu(D_n\cap A_{\varepsilon, C})>0.
$$
For any $x\in D_n\cap A_{\varepsilon, C}$ and any $0\not=u\in E_{ij}(x)$,
we have
\begin{equation}\label{fThmA4}
\begin{split}
&\|D(f^sg^th^r)_x^{2n}u\|
=\|Dh^{2stn^2}Df^{2sn}Dg^{2tn}Dh_x^{(2r-st)n}u\|  \\
\leq & l^8e^{4(s+t)n\varepsilon}le^{2sn\lambda_i(x, f)+2sn\varepsilon}
       le^{2tn\lambda_j(x,g)+2tn\varepsilon}le^{|2r-st|n\varepsilon}\|u\| \\
=& l^{11}\|u\|\exp n[2s\lambda_i(x, f)+2t\lambda_j(x, g)+\lambda\varepsilon]),
\end{split}
\end{equation}
where $\lambda=6s+6t+|2r-st|$.
In addition, from the definition of $A_{\varepsilon, C}$ in \eqref{fThmA1}
we get that for all $x\in A_{\varepsilon, C}$, there is
$0\not=u\in E_{ij}(x)$ such that
\begin{equation}\label{fThmA5}
\begin{split}
  &\|D(f^sg^th^r)_x^{2n}u\|>C^{-1}\|u\|
  \exp 2n[s\lambda_i(x,f)+t\lambda_j(x, g)+|\lambda|\varepsilon] \\
= & C^{-1}\exp (n|\lambda|\varepsilon)
   \|u\|\exp n[2s\lambda_i(x, f)+2t\lambda_j(x,g)+|\lambda|\varepsilon].
\end{split}
\end{equation}
Clearly, \eqref{fThmA4} and \eqref{fThmA5}
are contradict to each other if
$\disp n>\frac{{\rm log}(l^{11}C)}{|\lambda|\varepsilon}$.
Hence we must have $\mu(A^+_\varepsilon)=0$ for all
$\mu\in {\cal M}(M, \a)$ and and $\varepsilon>0$.
We complete the proof of Theorem A.
\end{proof}

\begin{proof}[Proof of Corollary~A.1]
By taking $s=t=0$ and $r=1$ in \eqref{fThmA} we know that the
Lyapunov exponent of any $0\ne u\in E_{ij}(x)$ is equal to $0$
with respect to $h$,  and so is that of any $0\ne u\in T_xM$.

Let $p$ be a periodic point of $h$ of period $n$, that is, $h^n(p)=p$.
Since $fh=hf$ and $gh=hg$, we get that both $f^n(p)$ and $g^n(p)$
are periodic orbits of $h$ with period $n$. Since there are only
finitely many periodic point of $h$ of period $n$, we get that
$\{f^sg^th^r(p): s,t, r\in {\mathbb Z}\}$ is a finite set. Hence
we can define a measure $\mu\in {\cal M}(M, \a)$ supported on the
set.  By finiteness and invariance we know that $\mu(\{p\})>0$,
i.e., $p$ is a generic point of $\mu$.  The fact that all Lyapunov
exponents of $h$ at $p$ are equal zero gives that the modulus of
all eigenvalues of $Dh^n(p)$ are equal to one.
\end{proof}

\begin{proof}[Proof of Corollary~A.2]
Assume that $h_{\top}(h)>0$, then there is some
$h$-invariant probability measure $\nu$ on $M$ such that the
metric entropy $h_\nu(h)>0$ by the variational principle. Consider
the probability measure sequence
$\nu_n\equiv\frac{1}{4n^2}\sum_{|i|,|j|\leq n}(f^ig^j)_*\nu$.
Passing to a subsequence if necessary,
suppose $\nu_n$ converges to a probability measure $\mu$ in the
weak-$*$ topology. It is easy to check that $\mu$ is
$\alpha(\H)$-invariant. 
Hence by Corollary~A.1 all Lyapunov exponents of $h$ are zero 
with respect to $\mu$.  Hence $h_{\mu}(h)=0$ by Ledrappier-Young's formula. 

On the other hand, since $h$ commutes with every element in $\alpha(\H)$, 
we have $h_{(f^ig^j)_*\nu}(h)=h_\nu(h)$ for any $i,j\in \Z$.
Since the entropy map $\nu\to h_{\nu}(h)$ is affine, 
we have $h_{\nu_n}(h)=h_\nu(h)$.  As the action
$\alpha$ is $C^\infty$, it follows from \cite{Nh} that
$0=h_{\nu}(h)\geq\lim\limits_{n\to\infty}h_{\mu_n}(h)=h_\mu(h)>0$.
This is a contradiction.
\end{proof}

\section{Faithfulness: Proof of Theorem B and its corollaries}\label{S3}
\setcounter{equation}{0}

Let $T$ be a diffeomorphism on a manifold $M$ with a hyperbolic set
$\Lambda$. For any $x\in \Lambda$, the \emph{stable manifold} of
$x$ for $T$ is defined by $W^s(x, T)=\{y\in M: d(T^nx,
T^ny)\rightarrow 0\ {\rm as}\ n \rightarrow \infty\}$, which is
$T$-invariant. For any $\varepsilon>0$, the \emph{local stable
manifold}  $W_\varepsilon^s(x, T)$ is the set $\{y\in M: d(T^nx,
T^ny)\leq\varepsilon \ {\rm for\ all}\ n\geq 0\}$. It is well
known that $W_\varepsilon^s(x, T)\subset W^s(x, T)$ and $W^s(x,
T)=\cup_{n\geq 0}T^{-n}W_\varepsilon^s(T^nx, T).$

\begin{lemma}\label{LThmB1}
Suppose $p$ is a common fixed point of $f$, $g$ and $h$, and $f$
is Anosov and has simple eigenvalues on stable direction with
$\l_->\l_+^2$ at $p$. Then either $h=\id$ or $h^2=\id$ on
$W^{s}(p,f)$.
\end{lemma}

\begin{proof}
Note that the eigenvalue of $Dh_p$ restricted to each stable
eigenspace is $\pm 1$ by Corollary~A.1.  We may assume it is $1$,
otherwise use $h^2$ instead of $h$. Since $f$ has simple eigenvalues 
on stable direction, and $h$ commutes with $f$, we must have
$Dh_p|_{E^s(p,f)}=\id$, where $E^s(p,f)=\{v\in T_p(M): \|Df_p(v)\|<\|v\|\}$.

Denote $r'=\min\{r, 2\}$.  Take $\varepsilon>0$ small enough such that
$\l_--\varepsilon>(\l_++\varepsilon)^{r'}$ and such that for any 
$x, y\in W_\varepsilon^s(p, f)$ and $n\in\mathbb N$,
\begin{equation}\label{fLemB1.1}
C_1(\lambda_{-}-\varepsilon)^nd(x, y)<d(f^n(x),
f^n(y))<C_2(\lambda_{+}+\varepsilon)^nd(x, y)
\end{equation}
for some fixed constants $C_1, C_2>0$. It is
clear that $hW^s(p, f)=W^s(p, f)$ by $hf=fh$. So there is $\e'\le \e$ 
such that $hW^s_{\e'}(p, f)\subset W^s_\e(p, f)$. 

Suppose $h(x)\neq x$ for some $x\in W^s_{\e'}(p,f)$. Let $x_n=f^n(x)$.
Then by \eqref{fLemB1.1} we have
\begin{equation}\label{fLemB1.2}
  \frac{d(x_n, h(x_n))}{d(x_n, p)}
= \frac{d(f^n(x), f^n(h(x)))}{d(f^n(x), p)}
\geq \frac{C_1}{C_2} \frac{(\l_--\varepsilon)^n d(x,h(x))}
{(\l_++\varepsilon)^n d(x,p)}.
  \end{equation}

Note that $W^s_{\e'}(p,f)$ is a $C^r$ submanifold 
tangent to $E^s(p,f)$ at $p$.
Take a local coordinate system on $W^s(p)$ at $p$.  We have
$$
h(x_n)-p=\int_0^1 Dh_{p+t(x_n-p)}(x_n-p) dt
=\Bigl( \id + (\int_0^1 Dh_{p+t(x_n-p)}dt -\id )\Bigr)(x_n-p).
$$
Since $h$ is a $C^r$ diffeomorphism and $Dh_p|_{E^s(p,f)}=\id$,
the equation gives
$$
\Bigl|\int_0^1 Dh_{p+t(x_n-p)}dt -\id \Bigr|\le C_3 |x_n-p|^{r'-1}
$$
for some $C_3>0$.  
Hence we get
$$
|h(x_n)-x_n|\le C_3 |x_n-p|^{r'}.
$$
Note that $|h(x_n)-x_n|=d(x_n, h(x_n))$ and $|x_n-p|=d(x_n,p)$.  
So by \eqref{fLemB1.1}
\begin{align*}
\frac{d(x_n, h(x_n))}{d(x_n, p)}
\leq C_3d(x_n,p)^{r'-1}
\leq C_3C_2(\l_++\varepsilon)^{n(r'-1)}d(x,p)^{r'-1}
\end{align*}
for all $n>0$, contradicting to \eqref{fLemB1.2} and
the fact $\l_--\varepsilon>(\l_++\varepsilon)^{r'}$. 

Then we must have $h(x)=x$ for any $x\in W^s_{\e'}(p,f)$, and then $h=\id$
on $W^{s}(p,f)$ by using the facts 
$W^{s}(p,f)=\cup_{n\geq 0}f^{-n}W_{\e'}^s(p,f)$ and $fh=hf$.
\end{proof}

\begin{lemma}\label{LThmB2}
Suppose $p$ is a periodic point of $f$ with period $n$ and $f$ has
only finitely many periodic points of period $n$. 
Then there are  $m, k\in \mathbb N$
such that $p$ is a common fixed point of $f^n$, $g^m$ and $h^k$.

In particular, if $p$ is the unique fixed point of $f$, then
$h(p)=p=g(p)$.
\end{lemma}

\begin{proof}
Since $fh=hf$, $h(p)$ is a periodic point of $f$ with period $n$.
By finiteness of $n$-periodic point set of $f$, there is some $k$
such that $h^k(p)=p$. Then we have
$f^ng^k(p)=g^kf^nh^{kn}(p)=g^k(p)$, that is, $g^k(p)$ is also a
periodic point of $f$ with period $n$.  This implies that
$g^{kl}(p)=g^{k}(p)$ for some $l\in\mathbb N$. Taking $m=lk-k$, 
we then complete the proof.

The second part of the lemma now is obvious.
\end{proof}

\begin{proof}[Proof of Theoren B]
Without loss of generality, we may suppose $f$ is an Anosov element
that has simple eigenvalues on stable direction with 
$\l_->\l_+^{\min\{r,2\}}$. 
By spectral decomposition, $f$ has basic sets $\O_1, \dots, \O_t$ (see
\cite{Bo}). On each basic set $\O_i$, we take a periodic point
$p_i\in \O_i$. Assume $f^{n_i}(p_i)=p_i$ for some $n_i\in\mathbb
N$.  Then there are $m_i$ and $k_i$ such that $p_i$ is a common
fixed point of $f^{n_i}, g^{m_i}$ and $h^{k_i}$ by Lemma~\ref{LThmB2}.
Applying Lemma~\ref{LThmB1} to $f^{n_i}, g^{m_ik_i}$ and $h^{n_im_ik_i}$,
we get $h^{2n_im_ik_i}=\id$ on $W^s(p_i, f^{n_i})$. Since
$M=\cup_{i=1}^t {\overline {W^s(p_i, f^{n_i})}}$, we get
$h^{2k}=\id$, where $k=\prod_{i=1}^t n_im_ik_i$.
\end{proof}

\begin{proof}[Proof of Corollary~B.2]
Since $f$ has only one fixed point,  we have $f(p)=g(p)=h(p)=p$ by
Lemma~\ref{LThmB2}.   Since $\dim E^s_p(f)=1$,
restricted to $E^s_p(f)$, $Df_p$ and $Dg_p$ commutes.
Hence, $Df_p\cdot Dg_p=Dg_p\cdot Df_p\cdot Dh_p$ implies 
$Dh_p|_{E^s_p(f)}=\id$.  
By the proof of Lemma~\ref{LThmB1}, we have that $h$ is identity
on $W^s(p, f)$. From \cite{Ne}, we know that $f$ is transitive.  So
$W^s(p, f)$ is dense in $M$, and $h$ is identity on $M$.
\end{proof}

\section{Affine Anosov action on tori: Proof of Theorem C \& D}\label{S4}
\setcounter{equation}{0}

Before the proof of Theorem C, let us recall two classical
results.

\begin{theorem}[Adler-Palais \cite{AP}]
If $R, S\in {\rm Aff}({\mathbb T^n})$ with $R$ being ergodic, then
any homeomorphism $\Phi$ of $\mathbb T^n$ with $\Phi R=S \Phi$ is
in ${\rm Aff}(\mathbb T^n)$.
\end{theorem}

\begin{theorem}[Franks-Manning \cite{Fr, Ma}]
Any Anosov diffeomorphism of $\mathbb T^n$ is topologically
conjugate to a hyperbolic toral automorphism.
\end{theorem}

\begin{proof}[Proof of Theoren C]
Suppose $k=f^rg^sh^t$ is Anosov for some $r, s, t\in\mathbb Z$.
Then by Theorem 4.2,  there is a homeomorphism $\Phi$ of $\mathbb
T^n$ such that $K=\Phi^{-1}k\Phi\in{\rm Aff}({\mathbb T^n})$.
Since $K$ is topologically transitive and affine, $K$ is ergodic.

Denote $F=\Phi^{-1}f\Phi$, $G=\Phi^{-1}g\Phi$ and $H=\Phi^{-1}h\Phi$.
Then we have  $FH=HF$, $GH=HG$, $FG=GFH$, and $K=F^rG^sH^t$.

Since $H^{-1}KH=K$ and $K\in \Aff(\T^n)$, $H\in {\rm Aff}(\mathbb
T^n)$ by Theorem 4.1. Similarly, since $F^{-1}KF=KH^{-s}$ and
$K,KH^{-s} \in \Aff(\T^n)$, $F\in {\rm Aff}(\mathbb T^n)$; and
since $G^{-1}KG=KH^{r}$ and $K,KH^{r} \in \Aff(\T^n)$, $G\in {\rm
Aff}(\mathbb T^n)$.
\end{proof}

\begin{lemma}\label{LThmD1}
Let $A, B, \mathcal{C}\in {\rm GL}(n, \mathbb R)$ such that $AB=BA\mathcal{C}, A\mathcal{C}=\mathcal{C}A$
and $B\mathcal{C}=\mathcal{C}B$. Suppose $A$ is hyperbolic with stable linear space
$E^s\subset \mathbb R^n$. If the modular of each eigenvalue of $\mathcal{C}$
is equal to 1, then $E^s$ is $B$ and $\mathcal{C}$ invariant.
\end{lemma}

\begin{proof}
For any $v\in E^s$, we have
$$\lim\limits_{n\to\infty}A^n\mathcal{C}v=\lim\limits_{n\to\infty}\mathcal{C}A^nv=0$$ 
by $A\mathcal{C}=\mathcal{C}A$, so $E^s$ is $\mathcal{C}$ invariant. 
Since the modular of each eigenvalue of $\mathcal{C}$
is equal to $1$, the increasing rate of matrix norm $\|\mathcal{C}^n\|$ is
bounded by a polynomial in $n$ by an easy calculation. Thus we
have
$$\lim\limits_{n\to\infty}A^nBv
=\lim\limits_{n\to\infty}BA^n\mathcal{C}^nv=\lim\limits_{n\to\infty}B\mathcal{C}^nA^nv=0$$
by $A\mathcal{C}=\mathcal{C}A$ and $AB=BA\mathcal{C}$. So $E^s$ is $B$  invariant.\end{proof}

\begin{lemma}\label{LThmD2}
Let $A, B, \mathcal{C}\in {\rm GL}(1, \mathbb R)$ such that $AB=BA\mathcal{C}$. Then
$\mathcal{C}={\rm \Id}$.
\end{lemma}

\begin{proof}
Since ${\rm GL}(1, \mathbb R)$ is commutative, we have
$AB=BA\mathcal{C}=AB\mathcal{C}$, which means that $\mathcal{C}$ is identity.
\end{proof}

\begin{lemma}\label{LThmD3}
Let $A, B, \mathcal{C}\in {\rm GL}(2, \mathbb R)$ such that $AB=BA\mathcal{C}, A\mathcal{C}=\mathcal{C}A$
and $B\mathcal{C}=\mathcal{C}B$.  If the modular of each eigenvalue of $\mathcal{C}$ is equal to
1, then $\mathcal{C}^2={\rm \Id}$.
\end{lemma}

\begin{proof}
Consider $A, B, \mathcal{C}$ as matrices in ${\rm GL}(2, \mathbb C)$.\vspace{2mm}

{\noindent\bf Claim 1.} The eigenvalues of $\mathcal{C}$ are $1$ or $-1$. In
fact, assume that  $\mathcal{C}$ has an eigenvalue $\lambda$ with
$\lambda^n\not=1$ for $n=1, 2$. By $A\mathcal{C}=\mathcal{C}A$, we can take a nonzero
vector $v\in\mathbb C^2$ such that $\mathcal{C}v=\lambda v$ and $Av=\gamma
v$ for some $\gamma\not=0$. Then we have
$$ABv=BA\mathcal{C}v=\lambda\gamma Bv \ \ {\mbox{and}}\ \ AB^2v=B^2A\mathcal{C}^2v=\lambda^2\gamma Bv.$$
So, $v, Bv$ and $B^2v$ are three eigenvectors of $A$ with pairwise
different eigenvalues, which is a contradiction. Hence $\lambda=1$
or $\lambda^2=1$, which means $\lambda=1$ or $-1$. \vspace{2mm}

{\noindent\bf Claim 2.} If $-1$ is an eigenvalue of $\mathcal{C}$, then
$\mathcal{C}^2={\rm \id}$.  In fact, we can take $v\in \mathbb {C}^2$ such that
$\mathcal{C}v=-v$ and $v, Bv$ are two eigenvectors of $A$ with different
eigenvalues as shown in Claim 1. So, under the base $\{v, Bv\}$,
$\mathcal{C}$ has the form $\begin{array}{ccc}
 \mathcal{C}=\disp\left[
  \begin{array}{ccc}
  -1 & 0 \\
  0 & -1
  \end{array}
\right],
\end{array}
$ and $\mathcal{C}^2={\rm \Id}$.\vspace{2mm}

{\noindent\bf Claim 3.} If all eigenvalues of $\mathcal{C}$ are 1, then
$\mathcal{C}={\rm \id}$.  In fact, if $A$ has two different eigenvalues, then
$\mathcal{C}$ is diagonal by $A\mathcal{C}=\mathcal{C}A$. So $\mathcal{C}$ can only be identity in this
case. Similarly, if $B$ has only simple eigenvalues, then $\mathcal{C}={\rm
\id}$. Thus we may suppose $A, B, \mathcal{C}$ have only eigenvalues
$\lambda, \gamma, 1$  respectively. If $A$ is diagonal, then $A$
and $B$ are commutative, and then $\mathcal{C}$ is identity by $AB=BA\mathcal{C}$. So we may suppose the eigenspace $V_\lambda$ of $A$ corresponding to
$\lambda$ is of dimension $1$. 
Fix an nonzero vector $v\in V_\lambda$.  
Then $\mathcal{C}V_\lambda=V_\lambda$ by $A\mathcal{C}=\mathcal{C}A$.
Therefore $ABv=BA\mathcal{C}v=\lambda Bv$ and we get $BV_\lambda=V_\lambda$.  
So $Bv=\gamma v$. Take a vector $w$ linearly independent of $v$. 
Then, under the base $\{v, w\}$, 
$A, B, \mathcal{C}$ has the forms
\[
\begin{array}{ccc}
A=\left[
  \begin{array}{ccc}
  \lambda & x\\
  0 & \lambda
  \end{array}
  \right], \quad
\ \ \ B=\left[
  \begin{array}{ccc}
  \gamma & y \\
  0 & \gamma
  \end{array}
\right],  \quad \ \ \ \mathcal{C}=\left[
  \begin{array}{ccc}
  1 & z \\
  0 & 1
  \end{array}
\right],
\end{array}
\]
for some $x, y, z\in\mathbb C$, which implies $\mathcal{C}={\rm \Id}$ by
$AB=BA\mathcal{C}$.
\end{proof}

\begin{proof}[Proof of Theorem D]
From Theorem B, there is a homeomorphism $\Phi$ on $\mathbb T^n$
such that $\Phi^{-1}f\Phi([x])=[Ax+a]$,
$\Phi^{-1}g\Phi([x])=[Bx+b]$ and $\Phi^{-1}h\Phi([x])=[\mathcal{C}x+c]$ for
any $x\in\mathbb R^n$, where $A, B, \mathcal{C}\in {\rm GL}(n, \mathbb Z)$
and $a, b, c\in \mathbb R^n$. It is easy to check that $AB=BA\mathcal{C},
A\mathcal{C}=\mathcal{C}A$ and $B\mathcal{C}=\mathcal{C}B$. Clearly $A$ is hyperbolic.
Let $E^s\subset\mathbb R^n$ be the stable linear subspace of $A$.
We assume ${\dim}(E^s)=1$ or $2$.

By Lemma~\ref{LThmD1}, $E^s$ is $A$, $B$ and $\mathcal{C}$ invariant.  
Since $0\in\mathbb T^n$ is a common fixed point of $A$, $B$ and
$\mathcal{C}$, we get that the modular of each eigenvalue of
$\mathcal{C}$ is $1$ by Corollary~A.1. Applying Lemma~\ref{LThmD2}
and \ref{LThmD3} to $A|_{E^s}$, $B|_{E^s}$ and
$\mathcal{C}|_{E^s}$, we know that $\mathcal{C}|_{E^s}=\Id$ if
${\dim}(E^s)=1$ and $\mathcal{C}|_{E^s}=\pm\Id$ if
${\dim}(E^s)=2$. It follows that $\mathcal{C}$ or $-\mathcal{C}$,
as automorphism of $\mathbb T^n$, is identity on $\mathbb T^n$ by
the density of $[E^s]$ in $\mathbb T^n$. Hence $\mathcal{C}$ or
$-\mathcal{C}$ is identity as matrix in ${\rm GL}(n, \mathbb Z)$.
Thus $\Phi^{-1}h\Phi([x])=[\pm x+c]$. So $h$ is conjugate to
either a translation or an affine transformation $T_{-\Id, c}$ for
some $c\in \T^n$, and the formal case occurs if $f$ is a
codimensional $1$ Anosov diffeomorphism.

Clearly, if $h$ is conjugate to $T_{-\Id, c}$, then $h^2=\id$. If
$h$ is conjugate to a translation $T_{\Id, c}$, then we can get
$[x+kc]=\Phi^{-1}h^k\Phi([x])=[x]$ for some $k>0$ and for any
$[x]\in\mathbb T^n$ by using the fact that $h$ sends a fixed point
of $f$ to a fixed point of $f$, and $f$ has only finite number of
fixed points.
\end{proof}

\begin{remark}
From the proof it is easy to see that the integer $k$ can be
chosen as a factor of the number of the fixed points of $f$.
\end{remark}

\section{Smooth Rigidity: Proof of Theorem E}\label{S5}
\setcounter{equation}{0}

\subsection{Setting of the problem and the KAM scheme} Before proceeding to specifics we will show how the general KAM scheme described in \cite[Section 3.3]{Damjanovic4} and \cite[Section 1.1]{Damjanovic3}
is adapted to the $\H$ action~$\alpha$.

\smallskip
\noindent \textbf{Step 1}. \emph{Setting up the linearized equation}

Let $\widetilde{\alpha}$ be a small perturbation of $\alpha$. To prove the existence
of a $C^\infty$ map $H$ such that $\widetilde{\alpha}\circ H=H\circ \alpha$, we need to solve the nonlinear conjugacy problem
\begin{align*}
  \alpha\circ\Omega-\Omega\circ\alpha=-R\circ(I+\Omega)
\end{align*}
where $\widetilde{\alpha}=\alpha+R$ and $H=I+\Omega$; and the corresponding linearized  conjugacy equation is
\begin{align}\label{for:10}
  \alpha\circ\Omega-\Omega\circ\alpha=-R
\end{align}
for small $\Omega$ and $R$.

Lemma \ref{le:2} shows that obtaining a $C^\infty$ conjugacy
for one ergodic generator suffices for the proof of Theorem E. Hence we just need to solve equation
\eqref{for:10} for one ergodic generator.

\smallskip
\noindent \textbf{Step 2}. \emph{Solving the linearized conjugacy equation for a particular
element.}

We classify the obstructions for solving the linearized equation \eqref{for:10} for an individual generator (see Lemma \ref{le:8} and \ref{le:5})
and obtain tame estimates are obtained for the solution. This means finite loss
of regularity in the chosen collection of norms in the Fr$\acute{e}$chet spaces,
such as $C^r$ or Sobolev norms.

\smallskip
\noindent \textbf{Step 3}. \emph{Constructing projection of the perturbation to the twisted cocycle space.}

First note that $R$ is a twisted cocycle not over $\alpha$ but over $\widetilde{\alpha}$ (see Lemma 3.3 of \cite{Damjanovic4}) thus \eqref{for:10} is not a twisted coboundary
equation over the linear action $\alpha$, just an approximation. Second is that even if \eqref{for:10} is a twisted coboundary over $\alpha$, it is impossible to produce a $C^\infty$ conjugacy for a single ergodic generator of the action. Therefore, we consider three generators,
and reduce the problem of solving the linearized equation \eqref{for:10} to solving
simultaneously the following system:
\begin{align}\label{for:11}
  A\circ\Omega-\Omega\circ A&=-R_A\notag\\
  B\circ\Omega-\Omega\circ B&=-R_B\notag\\
  \mathcal{C}\circ\Omega-\Omega\circ \mathcal{C}&=-R_\mathcal{C}
\end{align}
where $A$ and $B$ are ergodic generators and $\mathcal{C}$ is the center: $A:=\alpha(g_1)$, $B:=\alpha(g_2)$, $\mathcal{C}:=\alpha(g_3)$ and
$R_A:=R(g_1)$, $R_B:=R(g_2)$, $R_\mathcal{C}:=R(g_3)$.

As mentioned above, $R$ does not satisfy this twisted cocycle condition:
\begin{align}
 L(R_A,R_\mathcal{C})&\stackrel{\rm def}{=}\mathcal{C}R_A-R_A\circ \mathcal{C}-(AR_\mathcal{C}-R_\mathcal{C}\circ A)=0,\notag\\
 L(R_B,R_\mathcal{C})&\stackrel{\rm def}{=}\mathcal{C}R_B-R_B\circ \mathcal{C}-(BR_\mathcal{C}-R_\mathcal{C}\circ B)=0,\notag\\
 L(R_A,R_B)&\stackrel{\rm def}{=}R_A\circ B+AR_B-R_B\circ A\mathcal{C}-BR_\mathcal{C}\circ A-B\mathcal{C}R_A=0.
\end{align}
However the difference
\begin{align*}
 L(R_A,R_B), \qquad L(R_B,R_\mathcal{C})\quad\text{ and }\quad L(R_A,R_B)
\end{align*}
is quadratically small with respect to $R$ (see Lemma \ref{le:3}). More precisely, the perturbation $R$ can be split into two terms
\begin{align*}
 R=\mathcal{P}R+\mathcal{E}(R)
\end{align*}
so that $\mathcal{P}R$ is in the space of twisted cocycles and the error $\mathcal{E}(R)$ is bounded by the size of $L$ with the fixed loss of regularity (see Lemma \ref{le:4}).
More precisely, the system
\begin{align}
  -\mathcal{P}R_A&=-\big(R_A-\mathcal{E}(R_A)\big)=A\Omega-\Omega\circ A,\notag\\
   -\mathcal{P}R_B&=-\big(R_B-\mathcal{E}(R_B)\big)=B\Omega-\Omega\circ B,\notag\\
   -\mathcal{P}R_\mathcal{C}&=-\big(R_\mathcal{C}-\mathcal{E}(R_\mathcal{C})\big)=\mathcal{C}\Omega-\Omega\circ \mathcal{C}
\end{align}
has a common solution $\Omega$ after subtracting a part quadratically small to $R$.

\smallskip
\noindent \textbf{Step 4}. \emph{Conjugacy transforms the perturbed action into an action
quadratically close to the target.}

The common approximate solution $\Omega$ to the equations \eqref{for:11} above provides a new perturbation
\begin{align*}
\widetilde{\alpha}^{(1)}\stackrel{\rm def}{=}H^{-1}\circ\widetilde{\alpha}\circ H
\end{align*}
where $H=I+\Omega$, is much closer to $\alpha$ than $\widetilde{\alpha}$; i.e., the new error
\begin{align*}
 R^{(1)}\stackrel{\rm def}{=}\widetilde{\alpha}^{(1)}-\alpha
\end{align*}
is expected to be small with respect to the old error $R$.

\smallskip
\noindent \textbf{Step 5}. \emph{The process is iterated and the conjugacy is obtained.}

The iteration process is set and is carried out, producing a $C^\infty$ conjugacy which works for the action
generated by the three generators $A$, $B$ and $\mathcal{C}$. Ergodicity assures that it works
for all the other elements of the action $\alpha$.

\smallskip
What is described above highlights the essential features of the KAM scheme for the $\H$ action on torus.
The last two steps can follow
Section 5.2-5.4 in \cite{Damjanovic4} word by word without modification. Hence completeness of Step 2 and 3
admits the conclusion of Theorem E.

At the end of the this section, we prove a simple lemma which shows that obtaining a $C^\infty$ conjugacy
for one ergodic generator suffices for the proof of Theorem E.
\begin{lemma}\label{le:2}
Let $\alpha$ be a Heisenberg group $\mathcal{H}$ action by automorphisms of $\TT^N$ such that for
some $g\in \mathcal{H}$ the automorphism $\alpha(g)$ is ergodic. Let $\widetilde{\alpha}$ be a $C^1$ small perturbation
of $\alpha$ such that there exists a $C^\infty$ map $H:\TT^N\rightarrow\TT^N$ which is $C^1$ close to identity
and satisfies
\begin{align*}
 \widetilde{\alpha}(g)\circ H=H\circ\alpha(g).
\end{align*}
Then $H$ conjugates the corresponding maps for all the other elements of the action;
i.e., for all $h\in \mathcal{H}$ we have
\begin{align}\label{for:23}
 \widetilde{\alpha}(h)\circ H=H\circ\alpha(h).
\end{align}
\end{lemma}
\begin{proof}
Let $h$ be any element in $\mathcal{H}$ other than $g$. If $hg=gh$ it follows from \eqref{for:23} and
commutativity that
\begin{align*}
 &\alpha(g)\circ \tilde{h}=\tilde{h}\circ\alpha(g)
 \end{align*}
where $\tilde{h}=\alpha(h)\circ H^{-1}\circ\widetilde{\alpha}(h)^{-1}\circ H$.

If $hg=ghc$, where $c$ is the center of $\mathcal{H}$, then similarly we obtain
\begin{align*}
 &\alpha(g)\circ (\alpha(h)\circ H^{-1}\circ\widetilde{\alpha}(h)^{-1}\circ H)\\
 &=\alpha(h)\circ (\alpha(gc^{-1})\circ H^{-1})\circ\widetilde{\alpha}(h)^{-1}\circ H\\
 &\overset{(1)}{=}\alpha(h)\circ H^{-1}(\widetilde{\alpha}(gc^{-1})\circ\widetilde{\alpha}(h)^{-1})\circ H\\
 &=\alpha(h)\circ H^{-1}\widetilde{\alpha}(h^{-1})\circ(\widetilde{\alpha}(g)\circ H)\\
 &=(\alpha(h)\circ H^{-1}\widetilde{\alpha}(h^{-1})\circ H)\circ\alpha(g).
 \end{align*}
Here $(1)$ from the fact the $H$ also conjugates $\alpha(c)$ and $\widetilde{\alpha}(c)$ which is from previous analysis.

Then the conclusion follows immediately from the following fact (see Lemma 3.2 of \cite{Damjanovic4}): for any $C^1$ small enough map $F:\TT^N\rightarrow\TT^N$ , if $AF=F\circ A$, where $A\in GL(N,\ZZ)$ and is ergodic, then $F=0$.
\end{proof}

\subsection{Some notations and basic facts}

\begin{enumerate}
\item It is a result of Kronecker \cite{Kronecker} which
states that an integer matrix with all eigenvalues on the unit circle has to have
all eigenvalues roots of unity. Then there exists $n\in\NN$ such that all eigenvalues of $\mathcal{C}^n$ are $1$.
Using relation $AB=BA\mathcal{C}$, we obtain $AB^n=B^nA\mathcal{C}^n$. Hence we can assume that all eigenvalues of $\mathcal{C}$ are $1$, otherwise we
just turn to $A$, $B^n$ and $\mathcal{C}^n$ instead of $A$, $B$ and $\mathcal{C}$.

 \item The dual map $A^*$ on $\ZZ^N$ induces a decomposition of $\RR^N$ into expanding,
neutral and contracting subspaces. We will denote the expanding subspace by $V_1(A)$, the contracting subspace by $V_2(A)$ and the neutral subspace by $V_2(A)$.
\begin{align*}
 \RR^N=V_1(A)\bigoplus V_2(A) \bigoplus V_3(A).
\end{align*}
All three subspaces $V_i(A)$, $i=1$, $2$, $3$ are $A$ invariant and
\begin{align}
  \norm{A^iv}&\geq C\rho^i\norm{v}, &\rho>1,\quad &i\geq 0,\quad &v\in V_1(A), \notag\\
  \norm{A^iv}&\geq C\rho^{-i}\norm{v}, &\rho>1,\quad &i\leq 0,\quad &v\in V_3(A), \notag\\
  \norm{A^iv}&\geq C\abs{i}^{-N}\norm{v}, &\rho>1,\quad &i\neq 0,\quad &v\in V_2(A)\label{for:4}.
\end{align}

  \item For $v\in\ZZ^N$, $\abs{v}\overset{\text{def}}{=}\max\{\norm{\pi_1(v)},\,\norm{\pi_2(v)},\,\norm{\pi_3(v)}\}$ where $\norm{\cdot}$ is Euclidean
norm and $\pi_i(v)$ are projections of $v$ to subspaces $V_i$ ($i=1,\,2,\,3$) from \eqref{for:4},
that is, to the expanding, neutral, and contracting subspaces of $\RR^N$
for $A$ (or $B$); we will use the norm which is more convenient in a particular situation;
those are equivalent norms, the choice does not affect any results).

\smallskip
  \item For $v\in\ZZ^N$ we say $v$ is mostly in $i(A)$ for $i=1,\,2,\,3$ and will write $v\hookrightarrow i(A)$, if
the projection $\pi_i(v)$ of $v$ to the subspace $V_i$ corresponding to $A$ is sufficiently
large:
\begin{align*}
 \abs{v}=\norm{\pi(v)}.
\end{align*}
The notation $v\hookrightarrow 1,2(A)$ will be used for $v$ which is mostly in $1(A)$ or mostly
in $2(A)$.

\smallskip
\item Call $n\in\ZZ^N$ minimal
and denote it by $n_{\text{min}}$ if $v$ is the lowest point on its $A$ orbit in the sense that $n\hookrightarrow 3(A)$
and $An\hookrightarrow 1,2(A)$. There is one such
minimal point on each nontrivial dual $A$ orbit, we choose one on each dual $A$ orbit
and denote it by $n_{\text{min}}$. Then $n_{\text{min}}$ is substantially large both in $1,2(A)$ and in $3(A)$.

\smallskip

\item In what follows, $C$ will denote any constant that depends only on the given
linear action $\alpha$ with chosen generators $A$, $B$ and $\mathcal{C}$ and on the dimension of the
torus. $C_{x,y,z,\cdots}$ will denote any constant that in addition to the above depends
also on parameters $x$, $y$, $z$, $\cdots$.

\smallskip
\item Let $\theta$ be a $C^\infty$ function. Then we can write $\theta=\sum_{n\in\ZZ^N}\widehat{\theta}_ne_n$
where $e_n=e^{2\pi in\cdot x}$ are the characters. Then
\begin{itemize}
  \item [(i)] $\norm{\theta}_a\overset{\text{def}}{=}\sup_n\abs{\widehat{\theta}_n}\abs{n}^a$, $a>0$.

  \smallskip

  \item [(ii)] The following relations hold (see, for example, Section 3.1 of \cite{LLAVE}):
  \begin{align*}
   \norm{\theta}_r\leq C\norm{\theta}_{C^r},\qquad \norm{\theta}_{C^r}\leq C\norm{\theta}_{r+\sigma}
  \end{align*}
  where $\sigma>N+1$, and $r\in\NN$.

  \smallskip
  \item [(iii)] For any $F\in SL(N,\ZZ)$ $(\widehat{\theta\circ F})_n=\widehat{\theta}_{(F^\tau)^{-1}n}$ where $F^\tau$ denotes transpose matrix.
  We call $(F^\tau)^{-1}$ the dual map on $\ZZ^N$. To simplify the notation in the rest of the paper, whenever there is no
confusion as to which map we refer to we will denote the dual map by the same
symbol $F$.
\end{itemize}

\smallskip
\item For a map $\mathcal{F}$ with coordinate functions $f_i$ ($i=1,\cdots,k$) define $\norm{\mathcal{F}}_a\overset{\text{def}}{=}\max_{1\leq i\leq k}\norm{f_i}_a$. For two maps $\mathcal{F}$ and $\mathcal{G}$ define $\norm{\mathcal{F},\mathcal{G}}_a\overset{\text{def}}{=}\{\norm{\mathcal{F}}_a,\norm{\mathcal{G}}_a\}$.
$\norm{\mathcal{F}}_{C^r}$ and $\norm{\mathcal{F},\mathcal{G}}_{C^r}$ are defined similarly. For any $n\in \ZZ^N$
$\widehat{\mathcal{F}}_n\overset{\text{def}}{=}
((\widehat{f_1})_n,\cdots,(\widehat{f_k})_n)$.
\end{enumerate}

\subsection{Orbit growth for the dual action} In this section the crucial estimates for
the exponential growth along individual orbits of the dual action are obtained. The following follows directly from the proof of  Lemma 4.3 in \cite{Damjanovic4}:
\begin{lemma}\label{le:7} Let $Q_i$ be integer matrices in $SL(N,\ZZ)$, $1\leq i\leq m$ and suppose there exist constant $C,\,\tau>0$ such that for every non-zero integer vector $v\in\ZZ^N$ and for any $k=(k_1,\cdots,k_m)\in\ZZ^m$
\begin{align}\label{for:9}
\norm{Q_1^{k_1}\cdots Q_m^{k_m}n}\geq C\exp\{\tau\norm{k}\}\norm{n}^{-N}.
\end{align}
\begin{enumerate}

\smallskip

  \item [a)]\label{for:2} For any $C^{\infty}$ function $\varphi$ on the torus $\TT^N$ and any $y\in\CC$
the following sums:
\begin{align*}
  S_K(\varphi,n,y,p)=\sum_{k=(k_1,\cdots,k_m)\in K}y^{\norm{k}}\widehat{\varphi}_{Q_1^{k_1}\cdots Q_m^{k_m}n}
\end{align*}
converge absolutely for any $K\subset\ZZ^m$.
\smallskip

  \item [b)] Assume in addition to the assumptions in $a)$ that for a vector $n\in\ZZ^N$ and for
every $k=(k_1,\cdots,k_m)\in K=K(n)\subset\ZZ^m$ we have
\begin{align}\label{for:1111}
 p_1(\norm{k})\norm{Q_1^{k_1}\cdots Q_m^{k_m}n}\geq \norm{n}
\end{align}
where $p_1$ is a polynomial, then
\begin{align*}
\abs{S_K(\varphi,n,y,p)}\leq C_{a,y,\delta}\norm{\varphi}_a\norm{n}^{-a+\kappa_y+\delta}
\end{align*}
for any $a>\kappa_{y}\stackrel{\rm def}{=}\frac{N+1}{\tau}\abs{\log\abs{y}}$.
\smallskip

  \item [c)] If the assumptions \eqref{for:1111} is also satisfied for every $n\in\ZZ^N$, then the function
  \begin{align*}
   S(\varphi)\stackrel{\rm def}{=}\sum_{n\in\ZZ^N}S_{K(n)}(\varphi,n,y,p)e_n
  \end{align*}
is a $C^\infty$ function if $\varphi$ is. Moreover, the following norm comparison holds:
\begin{align*}
 \norm{S(\varphi)}_{C^r}\leq C_{r,y}\norm{\varphi}_{r+\sigma}
\end{align*}
for any $r\geq 0$ and $\sigma>N+2+[\kappa_{y}]$.

\end{enumerate}
\end{lemma}
\begin{corollary}\label{cor:2}
Suppose $Q_i,\,P_i\in \mathcal{H}$, $1\leq i\leq m$ and $K\subseteq \ZZ^m$. Set $y=\max_{1\leq i\leq m}\norm{P_i}$. If condition \eqref{for:9} is satisfied  for any $k\in K(n)$ then for any $C^{\infty}$ function $\varphi$ on the torus $\TT^N$ we obtain
\begin{enumerate}
  \item \label{for:26}the following sums:
\begin{align*}
  S_K(\varphi,n,P;Q)=\sum_{k=(k_1,\cdots,k_m)\in K(n)}P_1^{k_m}\cdots P_m^{k_1}\widehat{\varphi}_{Q_1^{k_1}\cdots Q_m^{k_m}n}
\end{align*}
converge absolutely, where $P$ stands for $P_1,\cdots,P_m$ and $Q$ stands for $Q_1,\cdots,Q_m$.

\smallskip
  \item \label{for:34}Assume in addition that for a vector $n\in\ZZ^N$ and for
every $k=(k_1,\cdots,k_m)\in K=K(n)\subset\ZZ^m$ we have
\begin{align*}
 p_1(\norm{k})\norm{Q_1^{k_1}\cdots Q_m^{k_m}n}\geq \norm{n}
\end{align*}
where $p_1$ is a polynomial, then
\begin{align*}
\abs{S_K(\varphi,n,P;Q)}\leq C_{a,y,\delta}\norm{\varphi}_a\norm{n}^{-a+\kappa_1+\delta}
\end{align*}
for any $a>\kappa_{P,Q}\stackrel{\rm def}{=}\frac{N+1}{\tau} \abs{\log\abs{y}}$.

\smallskip
\item \label{for:27}If the assumptions \eqref{for:1111} is satisfied for every $n\in\ZZ^N$, then the function
  \begin{align*}
   S(\varphi,P;Q)\stackrel{\rm def}{=}\sum_{n\in\ZZ^N}S_{K(n)}(\varphi,n,P,Q)e_n
  \end{align*}
is a $C^\infty$ function if $\varphi$ is. Moreover, the following norm comparison holds:
\begin{align*}
 \norm{S(\varphi)}_{C^r}\leq C_{r,y}\norm{\varphi}_{r+\sigma}
\end{align*}
for any $r\geq 0$ and $\sigma>N+2+[\kappa_{P,Q}]$.
\end{enumerate}
\end{corollary}
\begin{proof}
Since
\begin{align*}
 &\sum_{k=(k_1,\cdots,k_m)\in K(n)}\norm{P_1^{k_m}\cdots P_m^{k_1}\widehat{\varphi}_{Q_1^{k_1}\cdots Q_m^{k_m}n}}
 \leq \sum_{k=(k_1,\cdots,k_m)\in K(n)}y^{\norm{k}}\norm{\widehat{\varphi}_{Q_1^{k_1}\cdots Q_m^{k_m}n}},
\end{align*}
we get the conclusion immediately from above lemma.
\end{proof}
 In the subsequent
part we prove the exponential growth along individual orbits of ergodic elements. It may be viewed as a generalization of Lemma 4.3 in \cite{Damjanovic4} to higher rank non-abelian actions by toral
automorphisms.
\begin{lemma}\label{le:110}
There exist constant $C>0$ such that for every non-zero integer vector $v\in\ZZ^N$ and for any $k=(k_1,k_2)\in\ZZ^2\backslash 0$,
  \begin{align*}
\norm{A^{k_1}B^{k_2}v}\geq C\exp\{\tau(\abs{k_1}+\abs{k_2})\}\norm{v}^{-N}.
\end{align*}
\end{lemma}
\begin{proof}
From the Lyapunov space decomposition in Theorem A, we see that the proof of Lemma 4.3 in \cite{Damjanovic4} also applies to this case word by word. At first, we can  show  that there exists $\tau>0$ such that for any $k=(k_1,k_2)\in\ZZ^2\backslash 0$, there exists a Lyapunov space in which the Lyapunov exponent of $A^{k_1}B^{k_2}$ is greater than $\tau(\abs{k_1}+\abs{k_2})$. Ergodicity shows that the projection of $v$ to this space is greater than $\gamma\norm{v}^{-N}$, where $\gamma$ is a constant only dependent on the decomposition in Theorem A. Then we get the conclusion.
\end{proof}

\subsection{Twisted coboundary equation over a map on torus} Obstructions to solving a one-cohomology equation for a function over an ergodic toral automorphism
in $C^\infty$ category are sums of Fourier coefficients of the given function
along a dual orbit of the automorphism. This is the content of the Lemma 4.2 in \cite{Damjanovic4}. The same characterization holds however for
one-cohomology equation for a map over
ergodic toral automorphisms as well due to the estimate in Corollary \ref{cor:2}. The proofs of the two lemmas below follow closely the proof of Lemma 4.2 in \cite{Damjanovic4}
for solving a one-cohomology equation for functions.

\begin{lemma}\label{le:8}
Let $P$ and $Q$ are integer matrices in $SL(N,\ZZ)$ and $Q$ is ergodic. For a map $\theta$ on $\TT^N$, if there exists a $C^{\infty}$  map $\omega$ which is $C^0$ small enough on $\TT^N$ such that
\begin{align}\label{for:18}
 P\omega-\omega\circ Q=\theta,
\end{align}
then  the following sums along all nonzero dual orbits are zero, i.e.,
\begin{align*}
\sum_{i=-\infty}^\infty P^{-(i+1)}\hat{\theta}_{Q^iv}=0,\qquad \forall n\neq0.
\end{align*}
\end{lemma}
\begin{proof}
Since $\omega$ is $C^0$ small enough the equation \eqref{for:18} in the dual space has the form
\begin{align}\label{for:20}
 P\widehat{\omega}_n-\widehat{\omega}_{Qn}=\widehat{\theta}_n,\qquad \forall n\in\ZZ^N.
\end{align}
For any $m,\,\ell>0$ iterating the above equation  with respect to $A$ we get
\begin{align*}
  \sum_{i=-m}^{\ell}P^{-i}\widehat{\omega}_{Q^in}-\sum_{i=-m}^{\ell}P^{-(i+1)}\widehat{\omega}_{Q^{i+1}}=\sum_{i=-m}^{\ell}P^{-(i+1)}\widehat{\theta}_{Q^in},
\end{align*}
which simplifies to
\begin{align*}
  P^{m}\widehat{\omega}_{Q^{-m}n}-P^{-(\ell+1)}\widehat{\omega}_{Q^{\ell+1}n}=\sum_{i=-m}^{\ell}P^{-(i+1)}\widehat{\theta}_{Q^in}.
\end{align*}
Then the conclusion follows immediately if we can show that
\begin{align*}
 \lim_{m\rightarrow \infty}P^{m}\widehat{\theta}_{Q^{-m}n}=0,\qquad \forall n\neq 0,
\end{align*}
which is a direct consequence of  \eqref{for:26} of Corollary \ref{cor:2}.
\end{proof}

\begin{lemma}\label{le:5}
 Let $P$ and $Q$ be ergodic integer matrices in $SL(N,\ZZ)$.  Let $\theta$ be a $C^\infty$ map on the torus which is $C^\sigma$ small enough, where $\sigma>N+2+\kappa_{P^{-1},Q}$ ($\kappa_{P^{-1},Q}$ is defined in \eqref{for:34} of Corollary \ref{cor:2}).  If for all nonzero $n\in\ZZ$, the following sums along the dual orbits are zero, i.e.,
\begin{align}\label{for:1120}
\sum_{i=-\infty}^{\infty}P^{-(i+1)}\widehat{\theta}_{Q^in}=0,\qquad \forall n\neq 0.
\end{align}
Then the equation
\begin{align}\label{for:5}
  P\omega-\omega\circ Q=\theta
\end{align}
has a $C^\infty$ solution $\omega$, and the following estimate:
\begin{align}\label{for:8}
  \norm{\omega}_{C^r}\leq C_r\norm{\theta}_{C^{r+\sigma}}
\end{align}
for any $r\geq 0$, where $C_r$ is a number only dependent on Lyapunov exponents of $P$.
\end{lemma}
\begin{proof}
Suppose $\omega$ is a $C^\infty$ solution $C^0$ small enough to \eqref{for:5}. Then the equation \eqref{for:5} in the dual space has
the form
\begin{align}\label{for:6}
 P\widehat{\omega}_n-\widehat{\omega}_{Qn}=\widehat{\theta}_n,\qquad \forall n\in\ZZ^N.
\end{align}
For $n=0$, since $P$ is ergodic, we can immediately calculate $\widehat{\omega}_0=(P-I)^{-1}\widehat{\theta}_0$. For $n\neq 0$ the
dual equation has two solutions
\begin{align*}
  \widehat{\omega}_n^{\pm}=\pm\sum_{i\geq0 \atop i\leq-1}P^{-(i+1)}\widehat{\theta}_{Q^in},\qquad n\neq 0.
\end{align*}
Each sum converges absolutely by \eqref{for:26} of Corollary \ref{cor:2}. By assumption \eqref{for:1120}
$\widehat{\omega}_n^{+}=\widehat{\omega}_n^{-}\stackrel{\rm def}{=}\widehat{\omega}_n$. This gives a formal solution
$\omega=\sum\widehat{\omega}_n^{+}e_n=\sum\widehat{\omega}_n^{-}e_n$. We estimate
each $\widehat{\omega}_n$ using both of its forms in order to show that $\omega$ is $C^\infty$. In the notation of Corollary \ref{cor:2} we can write
\begin{align*}
 \widehat{\omega}^+_n=S_{K^{+}}(P^{-1}\theta,n,P^{-1},Q),\quad \text{ and }\quad \widehat{\omega}^-_n=-S_{K^{-}}(P^{-1}\theta,n,P^{-1},Q).
\end{align*}
Here $K^+=\{i\in\ZZ:i\geq 0\}$ and $K^-=\{i\in\ZZ:i\leq -1\}$.

If $n$ is mostly contracting, i.e., if $n\hookrightarrow 3(A)$, then
\begin{align}\label{for:14}
  \norm{A^in}\geq C\rho^{-i}\norm{n},\qquad \forall i\leq -1;
\end{align}
If $n$ is mostly contracting, i.e., if $n\hookrightarrow 1,2(A)$, then
\begin{align}\label{for:19}
  \norm{A^in}\geq Ci^{-N}\norm{n},\qquad \forall i\geq 0.
\end{align}
Thus the polynomial estimate needed for the application of part \eqref{for:34} of Corollary \ref{cor:2} is
satisfied either in $K^+$ and $K^-$ for any $n\in\ZZ^N$. This estimate implies that \eqref{for:8} holds. Finally, this also implies that smallness of $C^\sigma$ norm of $\theta$ guarantees $C^0$ smallness of $\omega$.

\end{proof}

\subsection{Construction of the projection}
\begin{lemma}\label{le:4}
Fix $\sigma=N+3+\kappa_{A^{-1},A}$. There exists $\delta>0$ such that for any $C^\infty$ maps $\theta$, $\psi$, $\omega$ on $\TT^N$ that are $C^\sigma$ small enough,
it is possible to split $\theta$, $\psi$ and $\omega$ as
\begin{gather*}
\theta=\Delta_A\Omega+\mathcal{R}\theta,\qquad \psi=\Delta_B\Omega+\mathcal{R}\psi\\
 \omega=\Delta_\mathcal{C}\Omega+\mathcal{R}\omega
 \end{gather*}
for a $C^\infty$ map $\Omega$, so that
\begin{align*}
 \norm{\mathcal{R}\theta,\mathcal{R}\psi,\mathcal{R}\omega}_{C^r}&\leq C_{r}\norm{R_1,R_2,R_3}_{C^{r+\delta}}\qquad\text{ and }\\
 \norm{\Omega}_{C^r}&\leq C_{r}\norm{\theta,\omega,\psi}_{C^{r+\sigma}}
 \end{align*}
for any $r\geq 0$, where
\begin{align}
R_1&\stackrel{\rm def}{=}\Delta_\mathcal{C}\theta-\Delta_A\omega,\label{for:16}\\
R_2&\stackrel{\rm def}{=}\Delta_\mathcal{C}\psi-\Delta_B\omega\label{for:35}
\end{align}
and
\begin{align}\label{for:3}
 R_3\stackrel{\rm def}{=}\theta\circ B+A\psi-\psi\circ A\mathcal{C}-B\omega\circ A-B\mathcal{C}\theta.
\end{align}
\end{lemma}
\begin{proof}
(1) \emph{Construction of $\Omega$ and $\mathcal{R}\theta$}. Let $\mathcal{R}\theta=\sum_n\widehat{\mathcal{R}\theta}_ne_n$ where
\begin{align*}
\widehat{\mathcal{R}\theta}_n&\stackrel{\rm def}{=}\left\{\begin{aligned} &\sum_{i\in\ZZ}A^{-i}\widehat{\mathcal{R}\theta}_{A^in},\qquad &n=n_{\text{min}},\\
&0, \qquad &\text{otherwise}
\end{aligned}
 \right.
\end{align*}
for $n\neq 0$ and $\widehat{\mathcal{R}\theta}_0\stackrel{\rm def}{=}0$. Let

Note that $n_{\text{min}}$ is substantially large both in the expanding and in the contracting
direction for $A$, then both \eqref{for:14} and \eqref{for:19} hold if $n=n_{\text{min}}$. The following estimate is obtained from \eqref{for:27} of Corollary \ref{cor:2}:
\begin{align}\label{for:31}
 \norm{\mathcal{R}\theta}_{C^r}\leq C_r\norm{\theta}_{C^{r+\sigma}},\qquad \forall r\geq 0.
\end{align}
Since $\theta-\mathcal{R}\theta$ satisfies the solvable condition in Lemma \ref{le:5}. By using Lemma \ref{le:5} there is a $C^\infty$ function $\Omega$ such that
\begin{align}\label{for:30}
 \Delta_A\Omega=\theta-\mathcal{R}\theta
\end{align}
with estimates
\begin{align*}
 \norm{\Omega}_{C^r}\leq C_r\norm{\theta-\mathcal{R}\theta}_{C^{r+\sigma}}\leq C_r\norm{\theta}_{C^{r+2\sigma}},\qquad \forall r\geq 0.
\end{align*}
(2) \emph{Estimates for $\mathcal{R}\theta$}. Rewrite \eqref{for:3} we get
\begin{align*}
 A\psi-\psi\circ A\mathcal{C}=B\omega\circ A+B\mathcal{C}\theta-\theta\circ B+R_3.
\end{align*}
Lemma \ref{le:8} shows that  the obstructions for $B\omega\circ A+B\mathcal{C}\theta-\theta\circ B+R_3$ with respect to $A\mathcal{C}$ vanish; therefore for any $n\neq 0$ we get
\begin{align*}
&\sum_iA^{-(i+1)}\widehat{\theta}_{B(A\mathcal{C})^in}\\
&=\sum_iA^{-(i+1)}B\mathcal{C}\widehat{\theta}_{(A\mathcal{C})^in}+\sum_iA^{-(i+1)}B\widehat{\omega}_{A(A\mathcal{C})^in}\\
&+\sum_iA^{-(i+1)}(\widehat{R_3})_{(A\mathcal{C})^in}
\end{align*}
since all the sums involved converge absolutely by \eqref{for:26} of Corollary \ref{cor:2}. Furthermore, by using the relation
\begin{align}\label{for:15}
  B(A\mathcal{C})^i=A^iB, \qquad \forall\,i\in\ZZ,
\end{align}
we obtain from the above relation
\begin{align}\label{for:37}
&\sum_iA^{-(i+1)}\widehat{\theta}_{A^iBn}-\sum_iBA^{-(i+1)}\widehat{\theta}_{A^in}\notag\\
&=\big(\sum_iA^{-(i+1)}B\mathcal{C}\widehat{\theta}_{(A\mathcal{C})^in}-\sum_iBA^{-(i+1)}\widehat{\theta}_{A^in}\big)\notag\\
&+\sum_iA^{-(i+1)}B\widehat{\omega}_{A(A\mathcal{C})^in}+\sum_iA^{-(i+1)}(\widehat{R_3})_{(A\mathcal{C})^in}.
\end{align}
Next, we will compute  the sum $\sum_iA^{-(i+1)}B\mathcal{C}\widehat{\theta}_{(A\mathcal{C})^in}-\sum_iBA^{-(i+1)}\widehat{\theta}_{A^in}$. To do so, we split it into two sums
$\sum_i=\sum_{i\geq 0}+\sum_{i\leq -1}$ and then use relation \eqref{for:16} to simplify each one. Set
\begin{align*}
  \Lambda=\Delta_A\omega.
\end{align*}
Then for any $n\neq 0$, we obtain from the proof of Lemma \ref{le:5}:
\begin{align}
 \sum_{i\geq 1}A^{-(i+1)}\widehat{\Lambda}_{A^in}&=\omega_n-A^{-1}\widehat{\Lambda}_{n}=A^{-1}\widehat{\omega}_{An}\qquad\text{ and}\label{for:13}\\
  -\sum_{i\leq-1}A^{-(i+1)}\widehat{\Lambda}_{A^in}&=\omega_n\label{for:17}.
\end{align}
Using relation \eqref{for:16} we get
\begin{align}\label{for:24}
&\widehat{\theta}_{A^in}-\mathcal{C}^{-i}\widehat{\theta}_{\mathcal{C}^iA^in}\notag\\
&=\left\{\begin{aligned} &\sum_{0\leq j\leq i-1}\mathcal{C}^{-(j+1)}(\widehat{\Lambda}_{\mathcal{C}^jA^in}-(\widehat{R_1})_{\mathcal{C}^jA^in}),\quad& i&\geq 1,\\
&-\sum_{i\leq j\leq -1}\mathcal{C}^{-(j+1)}(\widehat{\Lambda}_{\mathcal{C}^jA^in}-(\widehat{R_1})_{\mathcal{C}^jA^in}), \quad& i&\leq-1.
\end{aligned}
 \right.
\end{align}
Using \eqref{for:24} for the case of $i\geq 1$ we obtain
\begin{align*}
 &\sum_{i\geq 0}A^{-(i+1)}B\mathcal{C}\widehat{\theta}_{(A\mathcal{C})^in}-\sum_{i\geq 0}BA^{-(i+1)}\widehat{\theta}_{A^in}\\
 &=\sum_{i\geq 1}A^{-(i+1)}B\mathcal{C}\widehat{\theta}_{(A\mathcal{C})^in}-\sum_{i\geq 1}BA^{-(i+1)}\widehat{\theta}_{A^in}\\
 &\stackrel{(1)}{=}\sum_{i\geq 1}BA^{-(i+1)}\big(\mathcal{C}^{-i}\widehat{\theta}_{(A\mathcal{C})^in}-\widehat{\theta}_{A^in}\big)\\
 &=-\sum_{i\geq1}\sum_{j=0}^{i-1}BA^{-(i+1)}\mathcal{C}^{-(j+1)}(\widehat{\Lambda}_{\mathcal{C}^jA^in}-(\widehat{R_1})_{\mathcal{C}^jA^in})\\
 &=-\sum_{j\geq 0}\sum_{i\geq j+1}BA^{-(i+1)}\mathcal{C}^{-(j+1)}(\widehat{\Lambda}_{\mathcal{C}^jA^in}-(\widehat{R_1})_{\mathcal{C}^jA^in}).
\end{align*}
Here $(1)$ is from relation \eqref{for:15}. Of course, to justify the change of order of summation in the last equality, we must prove the absolute convergence of the sum. Using the notation in Corollary \ref{cor:2} we can write
\begin{align*}
 &\sum_{i\geq 0}A^{-(i+1)}B\mathcal{C}\widehat{\theta}_{(A\mathcal{C})^in}-\sum_{i\geq 0}BA^{-(i+1)}\widehat{\theta}_{A^in}\\
 &=BA\mathcal{C}S_K((\Lambda-R_1),n,A,\mathcal{C};\mathcal{C},A)
\end{align*}
where $K=\{(j,i)\in\ZZ^2:i-1\geq j\geq 0\}$. For any $i,\,j$ with $\abs{j}\leq\abs{i}$, \eqref{for:9} in Lemma \ref{le:7} shows that
\begin{align}\label{for:39}
 \abs{\mathcal{C}^jA^in}&\geq C\abs{j}^{-N}\abs{A^in}\geq C_1\abs{i}^{-N}\exp(\tau_A \abs{i})\abs{n}^{-N}\notag\\
 &\geq C_2\exp\{\tau_A (\abs{i}+\abs{j})/4\}\abs{n}^{-N},
\end{align}
where $C$, $C_1$ and $C_2$ are fixed numbers only dependent on $A$ and $\mathcal{C}$; and $\tau$ is defined in \eqref{for:9} of Lemma \ref{le:7}. This justifies to apply \eqref{for:26} of Corollary \ref{cor:2} to show the absolute convergence of the sum. Furthermore, we have
\begin{align*}
 &\sum_{j\geq 0}\sum_{i\geq j+1}B\mathcal{C}^{-(j+1)}A^{-(i+1)}\widehat{\Lambda}_{\mathcal{C}^jA^in}\\
 &=\sum_{j\geq 0}B\mathcal{C}^{-(j+1)}A^{-j}(\sum_{k\geq 1}A^{-(k+1)}\widehat{\Lambda}_{(A\mathcal{C})^jA^kn})\\
 &\stackrel{(1)}{=}\sum_{j\geq 0}B\mathcal{C}^{-(j+1)}A^{-(j+1)}\widehat{\omega}_{A(A\mathcal{C})^jn}\\
 &\stackrel{(2)}{=}\sum_{j\geq 0}A^{-(j+1)}B\widehat{\omega}_{A(A\mathcal{C})^jn}.
\end{align*}
Here $(1)$ follows from \eqref{for:13} and $(2)$ uses relation \eqref{for:15} again.

Hence we obtain
\begin{align}\label{for:38}
  &\sum_{i\geq 0}A^{-(i+1)}B\mathcal{C}\widehat{\theta}_{(A\mathcal{C})^in}-\sum_{i\geq 0}BA^{-(i+1)}\widehat{\theta}_{A^in}\notag\\
  &=-\sum_{j\geq 0}A^{-(j+1)}B\widehat{\omega}_{A(A\mathcal{C})^jn}\notag\\
  &+\sum_{j\geq 0}\sum_{i\geq j+1}BA^{-(i+1)}\mathcal{C}^{-(j+1)}(\widehat{R_1})_{\mathcal{C}^jA^in}.
\end{align}
To compute the sum $\sum_{i\leq-1}$ we use \eqref{for:24} for the case of $i\leq -1$:
\begin{align*}
 &\sum_{i\leq-1}A^{-(i+1)}B\mathcal{C}\widehat{\theta}_{(A\mathcal{C})^in}-\sum_{i\leq-1}BA^{-(i+1)}\widehat{\theta}_{A^in}\\
 &=\sum_{i\leq-1}BA^{-(i+1)}\big(\mathcal{C}^{-i}\widehat{\theta}_{(A\mathcal{C})^in}-\widehat{\theta}_{A^in}\big)\\
 &=\sum_{i\leq-1}\sum_{j=i}^{-1}BA^{-(i+1)}\mathcal{C}^{-(j+1)}(\widehat{\Lambda}_{\mathcal{C}^jA^in}-(\widehat{R_1})_{\mathcal{C}^jA^in})\\
 &=\sum_{j\leq-1}\sum_{i\leq j}BA^{-(i+1)}\mathcal{C}^{-(j+1)}(\widehat{\Lambda}_{\mathcal{C}^jA^in}-(\widehat{R_1})_{\mathcal{C}^jA^in}).
\end{align*}
Again we need to show the absolute convergence. We can also write
\begin{align*}
 &\sum_{i\leq-1}A^{-(i+1)}B\mathcal{C}\widehat{\theta}_{(A\mathcal{C})^in}-\sum_{i\leq-1}BA^{-(i+1)}\widehat{\theta}_{A^in}\\
 &=BA\mathcal{C}S_{K'}((\Lambda-R_1),n,A,\mathcal{C};\mathcal{C},A),
 \end{align*}
where $K'=\{(j,i)\in\ZZ^2:i\leq j\leq -1\}$. Then \eqref{for:39} shows that the absolute convergence follows from the same reason as in previous part.
Furthermore, by using \eqref{for:17} and relation \eqref{for:15} again we obtain
\begin{align*}
 &\sum_{j\leq-1}\sum_{i\leq j}B\mathcal{C}^{-(j+1)}A^{-(i+1)}\widehat{\Lambda}_{\mathcal{C}^jA^in}\\
 &=\sum_{j\leq-1}B\mathcal{C}^{-(j+1)}A^{-(j+1)}\big(\sum_{k\leq-1}A^{-(k+1)}\widehat{\Lambda}_{A(A\mathcal{C})^jA^kn}\big)\\
 &=-\sum_{j\leq-1}B\mathcal{C}^{-(j+1)}A^{-(j+1)}\widehat{\omega}_{A(A\mathcal{C})^jn}\\
 &=-\sum_{j\leq-1}A^{-(j+1)}B\widehat{\omega}_{A(A\mathcal{C})^jn}.
\end{align*}
Hence we obtain
\begin{align}\label{for:36}
  &\sum_{i\leq-1}A^{-(i+1)}B\mathcal{C}\widehat{\theta}_{(A\mathcal{C})^in}-\sum_{\leq-1}BA^{-(i+1)}\widehat{\theta}_{A^in}\notag\\
  &=-\sum_{j\leq-1}A^{-(j+1)}B\widehat{\omega}_{A(A\mathcal{C})^jn}\notag\\
  &-\sum_{j\leq-1}\sum_{i\leq j}BA^{-(i+1)}\mathcal{C}^{-(j+1)}(\widehat{R_1})_{\mathcal{C}^jA^in}.
\end{align}
By using \eqref{for:37}, \eqref{for:38} and \eqref{for:36} for any $n\neq 0$ we obtain:
\begin{align*}
 &\sum_iA^{-(i+1)}\widehat{\theta}_{A^iBn}-\sum_iBA^{-(i+1)}\widehat{\theta}_{A^in}\\
 &=\sum_{j\geq 0}\sum_{i\geq j+1}BA^{-(i+1)}\mathcal{C}^{-(j+1)}(\widehat{R_1})_{\mathcal{C}^jA^in}\\
 &-\sum_{j\leq-1}\sum_{i\leq j}BA^{-(i+1)}\mathcal{C}^{-(j+1)}(\widehat{R_1})_{\mathcal{C}^jA^in}\\
 &+\sum_iA^{-(i+1)}(\widehat{R_3})_{(A\mathcal{C})^in}.
\end{align*}
Iterating this equation with respect to $B$ we obtain
\begin{align*}
 &\sum_iA^{-(i+1)}\widehat{\theta}_{A^in}-\lim_{\ell\rightarrow\infty}\sum_iB^{-\ell}A^{-(i+1)}\widehat{\theta}_{A^iB^\ell n}\\
 &=-\sum_{k\geq 0}\sum_{j\geq 0}\sum_{i\geq j+1}B^{-k}A^{-(i+1)}\mathcal{C}^{-(j+1)}(\widehat{R_1})_{\mathcal{C}^jA^iB^kn}\\
 &+\sum_{k\geq 0}\sum_{j\leq-1}\sum_{i\leq j}B^{-k}A^{-(i+1)}\mathcal{C}^{-(j+1)}(\widehat{R_1})_{\mathcal{C}^jA^iB^kn}\\
 &-\sum_{k\geq 0}\sum_iB^{-(k+1)}A^{-(i+1)}(\widehat{R_3})_{(A\mathcal{C})^iB^kn}.
\end{align*}
The condition \eqref{for:9} in Corollary \ref{cor:2} is satisfied by
Lemma \ref{le:110}. Hence the limit above is $0$ from $(1)$ of Corollary \ref{cor:2}; and the absolute convergence of the sum involving $(A\mathcal{C})^iB^k$ is justified by the same reason. To show the absolute convergence of the other two sums involving $\mathcal{C}^jA^iB^k$ where $\abs{j}\leq\abs{i}$, by the same reason the following inequality is sufficient:
\begin{align*}
 \norm{\mathcal{C}^jA^iB^kn}&\geq C\abs{j}^{-N}\norm{A^iB^kn}\geq C\abs{i}^{-N}\exp\{\tau_{A,B}(\abs{i}+\abs{k})\}\norm{n}^{-N}\\
 &\geq C\exp\{\frac{1}{2}\tau_{A,B}(\abs{i}+\abs{k})\}\norm{n}^{-N}\\
 &\geq C\exp\{\frac{1}{4}\tau_{A,B}(\abs{i}+\abs{j}+\abs{k})\}\norm{n}^{-N}.
\end{align*}

Hence by the notation of Corollary \ref{cor:2} we obtain
\begin{align*}
 &\sum_iA^{-(i+1)}\widehat{\theta}_{A^in}\\
 &=-S_{K_1}((A\mathcal{C})^{-1}R_1,n,B^{-1},A^{-1},\mathcal{C}^{-1};\mathcal{C},A,B)\\
 &+S_{K_2}((A\mathcal{C})^{-1}R_1,n,B^{-1},A^{-1},\mathcal{C}^{-1};\mathcal{C},A,B)\\
 &-B^{-1}S_{K_3}(A^{-1}R_3,n,B^{-1},A^{-1};A\mathcal{C},B),
\end{align*}
where $K_1=\{(k_1,k_2,k_3)\in\ZZ^3:k_1\geq0,\,k_2\geq k_1+1,\,k_3\geq 0\}$, $K_2=\{(k_1,k_2,k_3)\in\ZZ^3:k_1\leq-1,\,k_2\leq k_1,\,k_3\geq 0\}$
and $K_3=\{(k_1,k_2)\in\ZZ^2:k_2\geq 0\}$.

By iterating backwards and applying the
same reasoning, we obtain
\begin{align*}
 &\sum_iA^{-(i+1)}\widehat{\theta}_{A^in}\\
 &=S_{K_1'}((A\mathcal{C})^{-1}R_1,n,B^{-1},A^{-1},\mathcal{C}^{-1};\mathcal{C},A,B)\\
 &-S_{K_2'}((A\mathcal{C})^{-1}R_1,n,B^{-1},A^{-1},\mathcal{C}^{-1};\mathcal{C},A,B)\\
 &+B^{-1}S_{K_3'}(A^{-1}R_3,n,B^{-1},A^{-1};A\mathcal{C},B),
\end{align*}
where $K_1'=\{(k_1,k_2,k_3)\in\ZZ^3:k_1\geq0,\,k_2\geq k_1+1,\,k_3\leq-1\}$, $K_2'=\{(k_1,k_2,k_3)\in\ZZ^3:k_1\leq-1,\,k_2\leq k_1,\,k_3\leq-1\}$
and $K_3'=\{(k_1,k_2)\in\ZZ^2:k_2\leq-1\}$.

Then according to \eqref{for:27} of Corollary \ref{cor:2}, the needed estimate for $\mathcal{R}\theta$
with respect to $R$ follows if in at least one of the union of half-spaces $K^+=\bigcup_{i=1}^3K_i$ and $K^-=\bigcup_{i=1}^3K_i'$ the
dual action satisfies some polynomial lower bound for every $n=n_{\text{min}}$.

Write $A$ and $B$ in block diagonal forms as stated in proof of Corollary \ref{cor:2}. In case $An\hookrightarrow 2(A)$, let $n_1$ be the largest projection of $n$ to some neutral block $J'$ for $A$. Then
\begin{align*}
  \norm{n_1}\geq C\norm{n}.
\end{align*}
Let the Lyapunov exponent of $B$ on this block be $\nu$. For all $j,\,i,\,k$ the Lyapunov exponent of $\mathcal{C}^jA^iB^k$ on $J'$ is $k\nu$.
Then if $\nu\geq 0$, on the half-space $K^+$ we obtain
\begin{align*}
 \norm{\mathcal{C}^jA^iB^kn}\geq C\abs{j}^{-N}\abs{i}^{-N}\abs{k}^{-N}\norm{n_1}\geq C_1\abs{ijk}^{-N}\norm{n},\quad \forall k\geq 0
\end{align*}
and
\begin{align*}
 \norm{\mathcal{C}^jA^iB^kn}\geq C\abs{j}^{-N}\abs{i}^{-N}\exp(k\nu/2)\norm{n_1}\geq C_1\abs{ij}^{-N}\norm{n}
\end{align*}
on $K^-$ for $k<0$ if $\nu< 0$.

Thus the polynomial estimate needed for the application of part \eqref{for:27} of Corollary \ref{cor:2} is
satisfied for such $n$.

In case $An\hookrightarrow 1(A)$, let $n_1$ and $n_2$ be the largest projections of $n$ to some blocks $J_1$ and $J_2$ with positive Lyapunov exponent $\lambda_1$
and negative Lyapunov exponent $\lambda_2$ respectively. Let $\nu_1$ and $\nu_2$ be corresponding Lyapunov exponents of $B$ on the two blocks. Then
\begin{align*}
  \norm{n_1}\geq C\norm{n},\qquad \norm{n_2}\geq C\norm{n}.
\end{align*}
For all $j,\,i,\,k$ the Lyapunov exponent of $\mathcal{C}^jA^iB^k$ on $J_1$ is $\chi(j,i,k)^+=i\lambda_1+k\nu_1$ and is $\chi(j,i,k)^-=i\lambda_2+k\nu_2$ on $J_2$. Next, we need to show that
\begin{align}\label{for:28}
 \{(j,i,k):\chi(j,i,k)^+\geq0\}\bigcup \{(j,i,k):\chi(j,i,k)^-\geq0\}
\end{align}
covers either $K^+$ or $K^-$.  This boils down to require $k(\frac{\nu_1}{\lambda_1}-\frac{\nu_2}{\lambda_2})\geq0$. Namely, for
any $(j,\,i)\in\ZZ^2$, $(j,\,i,\,k)$ belongs to the union in \eqref{for:28} if $k(\frac{\nu_1}{\lambda_1}-\frac{\nu_2}{\lambda_2})\geq0$
and this is true for $k\geq 0$ or for $k\leq 0$ depending on the sign of $\frac{\nu_1}{\lambda_1}-\frac{\nu_2}{\lambda_2}$. Therefore we obtain
\begin{align}\label{for:29}
 \norm{\mathcal{C}^jA^iB^kn}\geq C\abs{ijk}^{-N}\norm{n}.
\end{align}
in $K^+$ or in $K^-$.

Now choose the half-space in which the estimate \eqref{for:29} holds, that is choose
one of the sums $-S_{K_1}+S_{K_2}-S_{K_3}$ or $S_{K_1'}-S_{K_2'}+S_{K_3'}$. Then the assumptions of \eqref{for:27} of Corollary \ref{cor:2} are satisfied for one of the sums above; and therefore the estimate
for follows:
\begin{align}\label{for:41}
 \norm{\mathcal{R}\theta}_{C^r}\leq C_{r}\norm{R_1,R_2,R_3}_{r+\sigma_1}
\end{align}
for any $r\geq 0$ and $\sigma_1>N+2+[\kappa_{y}]$, where $\kappa_{y}=\frac{4(N+1)}{\tau_{A,B}}\abs{\log y}$, where $y=\max\{\norm{A},\norm{B}\}$.

\smallskip
\noindent (3) \emph{Estimates for $\mathcal{R}\omega$}. By  \eqref{for:16} and \eqref{for:30} we obtain
\begin{align}\label{for:32}
\Delta_A(\omega-\Delta_\mathcal{C}\Omega)&=\Delta_\mathcal{C}\mathcal{R}\theta-R_1.
\end{align}
Define
\begin{align*}
  \mathcal{R}\omega&\stackrel{\rm def}{=}\omega-\Delta_\mathcal{C}\Omega\qquad\text{ and }\notag\\
  \mathcal{R}\psi&\stackrel{\rm def}{=}\psi-\Delta_B\Omega.
\end{align*}
By using Lemma \ref{le:5} to equation \eqref{for:32} we obtain
\begin{align}\label{for:42}
\norm{\mathcal{R}\omega}_{C^r}\leq C_r\norm{\Delta_\mathcal{C}\mathcal{R}\theta-R_1}_{C^{r+\sigma}}\leq C_r\norm{R_1,\,R_2,\,R_3}_{C^{r+\sigma+\sigma_1}}
\end{align}
for any $r\geq 0$.

\smallskip
\noindent(3) \emph{Estimates for $\mathcal{R}\psi$} In \eqref{for:3}, substituting $\omega$ by $\Delta_\mathcal{C}\Omega+\mathcal{R}\omega$ and $\psi$ by $\Delta_B\Omega+\mathcal{R}\psi$ we get
\begin{align*}
 A\mathcal{R}\psi-\mathcal{R}\psi\circ A\mathcal{C}=R_3+\mathcal{R}\theta\circ B+B\mathcal{C}\mathcal{R}\theta+B\mathcal{R}\omega\circ A.
\end{align*}
Again, Lemma \ref{le:5} implies that
\begin{align}\label{for:40}
 \norm{\mathcal{R}\psi}_{C^r}&\leq C_{r}\norm{R_3+\mathcal{R}\theta\circ B+B\mathcal{C}\mathcal{R}\theta+B\mathcal{R}\omega\circ A}_{r+\sigma_3}\notag\\
 &\leq C_r\norm{R_1,R_2,R_3}_{C^{r+\sigma+\sigma_1+\sigma_2+\sigma_3}}
\end{align}
where $\sigma_3>N+2+[\kappa_{A^{-1},A\mathcal{C}}]$.

Let $\delta=\sigma+\sigma_1+\sigma_2+\sigma_3$. Then the conclusion follows from \eqref{for:31}, \eqref{for:41}, \eqref{for:42} and \eqref{for:40}.
\end{proof}
In fact $\theta$, $\psi$, $\omega$ play the roles of $R_A$, $R_B$ and $R_\mathcal{C}$ in \eqref{for:11}. The following lemma shows that $R_1,\,R_2,\,R_3$ cannot be large if $\alpha+R$ is a Heisenberg group action. It is in fact quadratically small with respect to $R$.
\begin{lemma}\label{le:3}
If $\widetilde{\alpha}=\alpha+R$ is a $C^\infty$ Heisenberg group action on $\TT^N$ then for any $r\geq 0$
\begin{align}\label{for:44}
 \norm{R_1,R_2,R_3}_{C^r}\leq C_r\norm{\theta,\psi,\omega}_{C^r}\norm{\theta,\psi,\omega}_{C^{r+1}}.
\end{align}
\end{lemma}
\begin{proof}
The estimates for $R_1$ an $R_2$ follow the same way as in the proof of Lemma 4.7 in \cite{Damjanovic4}. We just need to show the estimate for $R_3$. Note that
\begin{align*}
  \widetilde{\alpha}_A\circ \widetilde{\alpha}_B&=\widetilde{\alpha}_B\circ \widetilde{\alpha}_{\mathcal{C}}\circ \widetilde{\alpha}_{A}\\
  (A+\theta)\circ (B+\psi)&=(B+\psi)\circ (\mathcal{C}+\omega)\circ (A+\theta).
  \end{align*}
Then
\begin{align*}
 &\theta\circ B+A\psi\\
 &=\psi\circ \big(\omega\circ(A+\theta)+\mathcal{C}A+\mathcal{C}\theta\big)+B\omega\circ(A+\theta)+B\mathcal{C}\theta\\
 &+\theta\circ B-\theta\circ (B+\psi).
\end{align*}
Therefore,
\begin{align*}
  R_3&=\theta\circ B+A\psi-\psi\circ A\mathcal{C}-B\omega\circ A-B\mathcal{C}\theta\\
  &=\psi\circ \big(\omega\circ(A+\theta)+\mathcal{C}A+\mathcal{C}\theta\big)-\psi\circ A\mathcal{C}\\
 &+\theta\circ B-\theta\circ (B+\psi)\\
 &+B\omega\circ(A+\theta)-B\omega\circ A.
\end{align*}
The estimate \eqref{for:44} for $C^r$ norms follows similarly (see for example [\cite{LAZUTKIN}, Appendix II]):
\begin{align*}
 \norm{R_3}_{C^r}&\leq C_r\norm{\psi,\omega\circ(A+\theta)+\mathcal{C}\theta}_{C^r}\norm{\psi,\omega\circ(A+\theta)+\mathcal{C}\theta}_{C^{r+1}}\\
 &+C_r\norm{\psi,\theta}_{C^r}\norm{\psi,\theta}_{C^{r+1}}+C_r\norm{\omega,\theta}_{C^r}\norm{\omega,\theta}_{C^{r+1}}\\
 &\leq C_r\norm{\theta,\psi,\omega}_{C^r}\norm{\theta,\psi,\omega}_{C^{r+1}}.
\end{align*}
\end{proof}

\end{document}